\documentclass{scrartcl}
\usepackage[T1]{fontenc}
\usepackage[english]{babel}
\usepackage[latin1]{inputenc}
\usepackage{amsmath}
\usepackage{amsthm,amssymb}
\usepackage{dsfont}
\usepackage{csquotes}
\usepackage{graphicx}
\usepackage{color, pgf}
\usepackage{hyperref}
\usepackage{url}
\usepackage{booktabs}
\usepackage{multirow}
\hypersetup{colorlinks=true, urlcolor=blue, citecolor=red} 
\newtheorem{thm}{Theorem}[section]
\newtheorem{defn}[thm]{Definition}
\newtheorem{lem}[thm]{Lemma}
\newtheorem{cor}[thm]{Corollary}
\newtheorem{rem}[thm]{Remark}
\newtheorem{prop}[thm]{Proposition}

\newenvironment{Assumptions}
{
\setcounter{enumi}{0}

\begin{enumerate}}
{\end{enumerate} }

\newcommand{\edgesint}{\mathcal{E}_{\operatorname{int}}}
\newcommand{\reg}{\operatorname{reg}(\mathcal{T})}
\newcommand{\Tau}{\mathcal{T}}
\newcommand{\dkl}{d_{K|L}}

\newcommand{\E}{\mathbb{E}}
\newcommand{\N}{\mathbb{N}}
\newcommand{\R}{\mathbb{R}}
\newcommand{\erww}[1]{\mathbb{E}\left[#1\right]}

\begin{document}

\title{Convergence rates for a finite volume scheme of a stochastic non-linear parabolic equation}
\author{Kavin Rajasekaran \thanks{Institute of Mathematics, Clausthal University of Technology, 38678 Clausthal-Zellerfeld, Germany \href{mailto:aleksandra.zimmermann@tu-clausthal.de}{\texttt{kavin.rajasekaran@tu-clausthal.de}}} \and Niklas Sapountzoglou \thanks{Institute of Mathematics, Clausthal University of Technology, 38678 Clausthal-Zellerfeld, Germany \href{mailto:niklas.sapountzoglou@tu-clausthal.de}{\texttt{niklas.sapountzoglou@tu-clausthal.de}}}}
\date{}

\maketitle

\begin{abstract}
In this contribution, we provide convergence rates for a finite volume scheme of a stochastic non-linear parabolic equation with multiplicative Lipschitz noise and homogeneous Neumann boundary conditions. More precisely, we give an error estimate for the $L^2$-norm of the space-time discretization by a semi-implicit Euler scheme with respect to time and a two-point flux approximation finite volume scheme with respect to space and the variational solution. The only regularity assumptions additionally needed is spatial regularity of the initial datum and smoothness of the diffusive term.
\end{abstract}
\noindent
\textbf{Keywords:} Finite volume scheme $\bullet$ Error estimates 
$\bullet$ Stochastic non-linear parabolic equation $\bullet$ 
Multiplicative noise $\bullet$ Space-time discretization \\

\noindent
\textbf{Mathematics Subject Classification (2020):} 60H15 $\bullet$ 35K05 $\bullet$ 65M08

\section{Introduction}
Let $\Lambda$ be a bounded, open, connected, and polygonal set of $\mathbb{R}^d$ with $d=2,3$.
Moreover let $(\Omega,\mathcal{A},\mathds{P})$ be a probability space endowed with a right-continuous, complete filtration $(\mathcal{F}_t)_{t\geq 0}$ and let $(W(t))_{t\geq 0}$ be a standard, one-dimensional Brownian motion with respect to $(\mathcal{F}_t)_{t\geq 0}$ on $(\Omega,\mathcal{A},\mathds{P})$. For a measurable space $(\mathcal{S},\mu)$, a separable Banach space $X$ and $1\leq \nu \leq \infty$, we will denote the classical Lebesgue-Bochner space by $L^{\nu}(S;X)$. If $D \subset \R^d$ is open and $k \in \N$ we will write $H^k(D):=\{v \in L^2(D);~\partial_{\alpha} v \in L^2(D)~\forall \, \alpha \text{ with } |\alpha| \leq k\}$ in the sequel.
For $T>0$, we consider a stochastic non-linear parabolic equation forced by a multiplicative stochastic noise:
\begin{align}\label{equation}
\begin{aligned}
du-\Delta u\,dt + \mbox{div}_x (\mathbf{v}f(u))\, dt &=g(u)\,dW(t)+\beta(u)\,dt, &&\text{ in }\Omega\times(0,T)\times\Lambda;\\
u(0,\cdot)&=u_0, &&\text{ in } \Omega\times\Lambda;\\
\nabla u\cdot \vec{n}&=0, &&\text{ on }\Omega\times(0,T)\times\partial\Lambda;
\end{aligned}
\end{align}
where $\vec{n}$ denotes the unit normal vector to $\partial\Lambda$ outward to $\Lambda$, $u_0\in L^2(\Omega;L^2(\Lambda))$ is $\mathcal{F}_0$-measurable and $\Delta$ denotes the Laplace operator on $H^1(\Lambda)$ associated with the weak formulation of the homogeneous Neumann boundary condition and $\operatorname{div}$ denotes the divergence operator w.r.t. space. Moreover, the stochastic noise term $g$, the external force term $\beta$ and the term $f$ are real-valued Lipschitz functions and $\mathbf{v}$ is defined on $[0,T] \times \Lambda$ taking values in $\R^d$. Hence, the term $\mathbf{v} f(u)$ denotes the flux of the system.

\subsection{State of the art}
Existence and uniqueness results of variational solutions, mild solutions, weak solutions and pathwise solutions of Problem \eqref{equation} is well known in the literature, see e.g., \cite{DPraZab, KryRoz81, LR, PardouxThese}. In this paper, we will focus on variational solutions defined in Defintion \ref{solution}. In the last decades, the interest in studying numerical schemes for SPDEs arised in the literature. A brief overview of the state of the art is given by \cite{ACQS20, DP09, OPW22}. In the past, finite-element methods have been extensively used and have been favored in the numerical analysis of variational solutions, see e.g., \cite{BBNP14, BHL21} and the cited papers therein. Finite volume schemes can preserve important physical properties at the numerical level. 
Droniou pointed out in \cite{D14} that the finite volume schemes are favorable amongst other numerical schemes when it comes down to the conservation of properties of the underlying equation. The TPFA scheme is a cell-centered finite volume scheme and cheap to implement due to the fact that the matrices are sparse. Especially when implementing numerical schemes of SPDEs instead of PDEs, a numerical scheme which is cheap to implement and fast running is very valuable. Finite volume schemes have been studied for stochastic scalar conservation laws in, \textit{e.g.}, \cite{BCC20, BCG16_01, BCG16_02, BCG17}. \\
\newline
Moreover, an existence and uniqueness result of the solution to a finite volume scheme of Problem \eqref{equation} can be found in \cite{BNSZ23_2}.\\
\newline
%In the literature, the existence and uniqueness of variational solutions to Problem \eqref{equation} (see Definition \ref{solution}) is well known and is covered by the classical framework about stochastic parabolic equations, see, \textit{e.g.}, \cite{LR, KryRoz81, PardouxThese, DPraZab}. We are interested in studying numerical schemes for parabolic stochastic partial differential equations (SPDEs) of type \eqref{equation}.\\
Concerning elliptic and parabolic second-order problems, \cite{ABH07} points out the flexibility on the geometry of the meshes and ensure the local consistency of the numerical fluxes inside the domain. The finite volume discretization is well-suited when considering an additional convective first-order term to the equation. Standard cell-centered finite volume methods may be used when the leading second-order operator in the equation is the Laplacian (see, e.g., \cite{EGH00}) such as the two point flux approximation scheme (TPFA).
General linear, elliptic, second-order deterministic equations have been studied in e.g. \cite{CVV99, DE06, EGH06}. Therein, the gradient reconstruction procedure requires more than the normal direction between two centers of neighbouring cells and the finite volume scheme becomes more involved, from the theoretical and from the numerical point of view. In this contribution, we restricted ourselves to the case of the Laplacian and the TPFA scheme. In the case $\operatorname{div}(\mathbf{v}f(u))=0$, the convergence of a space-time discretization by a finite volume scheme towards the unique variational solution of the stochastic heat equation in two spatial dimensions has been studied in \cite{BN20} in the case of a linear multiplicative noise and in \cite{BNSZ23} in the more general case of a multiplicative Lipschitz noise. Therein, a TPFA has been used with respect to space and a semi-implicit Euler scheme has been used for the time discretization. In \cite{BN20}, the convergence has been proven by classical arguments. In \cite{BNSZ23}, the convergence of the nonlinear multiplicative noise term has been addressed using the stochastic compactness method based on the theorems of Prokhorov and Skorokhod. The authors in \cite{SZ25} have shown convergence rates for a finite volume scheme of the stochastic heat equation. For general $f$ and divergence-free $\mathbf{v}$, the existence of a solution to a finite volume scheme of Problem \eqref{equation} has been shown in \cite{BNSZ23_2}. The convection term is approximated by an upwind scheme. In \cite{BSZ24}, the convergence of the solution of the finite volume scheme in \cite{BNSZ23_2} to the variational solution could be shown in the case $f \equiv id$ without using the stochastic compactness method, making this approach more accessible without deeper knowledge of stochastic analysis. In the case $\operatorname{div}(\mathbf{v}f(u)) \neq 0$, convergence of a finite volume scheme to the variational solution has not been proven yet. Even though if additionally $f \equiv id$, for all these finite volume schemes, convergence has been proven without giving any rate. To the best of our knowledge, the only result about convergence rates of finite volume schemes for parabolic SPDEs in the variational setting is the work \cite{SZ25}, where the case $\operatorname{div}(\mathbf{v}f(u)) = 0$ is considered. However, there are many results concerning convergence rates for different numerical schemes of SPDEs, see, \textit{e.g.}, \cite{BP23, KG22, MajPro18} and, in particular, see \cite{BHL21, DHW23, GM05, GM07} for space-time discretizations and convergence rates for nonlinear monotone stochastic evolution equations. In our study, we want to provide error estimates for the finite volume scheme proposed in \cite{BNSZ23_2} in the case $f \equiv id$ in space dimensions $d=2$ and $d=3$ under rather mild and natural regularity assumptions on the initial condition $u_0$ and on the diffusive term $g$ in the stochastic integral.

\subsection{Aim of the study}
We want to provide a generalization of the result in \cite{SZ25} to the case $\operatorname{div}(\mathbf{v}f(u)) \neq 0$, where $f \equiv id$ and $\mathbf{v}$ is divergence free. However, we tried to provide a result for general $f$, but we did not manage to get the last argument right in the general case.
In order to be precise, we want to provide a result about convergence rates of a finite volume scheme for parabolic SPDEs with homogeneous Neumann boundary conditions. We consider the case of two and three spatial dimensions. The proposed finite volume scheme does not coincide with the finite element methods and other numerical schemes for which convergence rates for \eqref{equation} have been considered. Our study provides error estimates for the $L^2$-norm of the space-time discretization of \eqref{equation} proposed in \eqref{FVS} and the variational solution of \eqref{equation} of order $O(\tau^{1/2} + h + h \tau^{- 1/2})$, where $\tau$ represents the time step and $h$ the spatial parameter. The unsatisfying term $h \tau^{-1/2}$ arises due to the fact that SPDEs are considered and differentiability in time cannot be assumed for the solution itself. In the deterministic case, see e.g. \cite{Flore21}, one may get rid of the unsatisfying term $h \tau^{-1/2}$ by assuming the exact solution to be two times differentiable in time. In the stochastic case this unsatisfying term appears also in other cases, e.g. for mixed finite element schemes for the stochastic Navier-Stokes equation (see \cite{FQ21}). We compare the exact solution of Problem \eqref{equation} with the solution of the semi-implicit Euler scheme first. $H^2$-regularity in space for the initial value $u_0$ and smoothness for the function $g$ in the stochastic It\^{o} integral, $\mathbf{v}$ and $\beta$ has to be assumed. Then, the error is of order $O(\tau^{1/2})$. Finally, comparing the centered projection of the solution of the semi-implicit Euler scheme with the solution of the finite volume scheme provides an error of order $O(\tau + h + h \tau^{-1/2})$.
\subsection{Outline}
The paper is organized as follows:\\
In Section 2, we introduce the considered problem as well as the finite volume framework as well as some analytic tools which are used frequently.\\
In Section 3, we state the main result and the regularity assumptions.\\
In Section 4, we introduce the semi-implicit Euler scheme for a stochastic non-linear parabolic equation with Neumann boundary conditions and we provide a-posteriori estimates for its solution.\\
Section 5 contains stability estimates for the exact solution of a stochastic non-linear parabolic equation and a regularity result.\\
In Section 6, we estimate the $L^2$-error between the exact solution and the solution of the semi-implicit Euler scheme and the $L^2$-error between the centered projection of the solution of the semi-implicit Euler scheme and the solution of the finite volume scheme.\\
In Section 7, we collect the previous results and prove our main theorem.\\
%Finally, we present our computational experiments in Section 8.
\section{A Finite Volume scheme for a stochastic non-linear parabolic equation}
\subsection{A stochastic non-linear parabolic equation with multiplicative noise}
For the well-posedness of a variational solution to a stochastic non-linear parabolic equation, we assume the following regularity on the data:
\begin{itemize}
\item[$i.)$] $u_0\in L^2(\Omega;L^2(\Lambda))$ is $\mathcal{F}_0$-measurable.
\item[$ii.)$] $g:\mathbb{R}\rightarrow\mathbb{R}$ is a Lipschitz continuous function.
\item[$iii.)$] $f:\mathbb{R}\rightarrow\mathbb{R}$ is non-decreasing and Lipschitz continuous function with $f(0)=0$.
\item[$iv.)$] $\beta:\mathbb{R}\rightarrow\mathbb{R}$ is a Lipschitz continuous function with $\beta(0)=0$.
\item[$v.)$] $\mathbf{v} \in C^1([0,T] \times \Lambda; \mathbb{R}^d)$ such that ${\operatorname{div}} (\mathbf{v})=0$ in $[0,T] \times \Lambda$ and $ \mathbf{v} \cdot \vec{n}=0$ on $[0,T] \times \partial \Lambda$.\\
\end{itemize}
For $u_0$ and $g$ that satisfy the above assumptions, we will be interested in the following concept of \textit{variational solution} for \eqref{equation}:
\begin{defn}\label{solution} A predictable stochastic process $u$ in $L^2\left(\Omega\times(0,T);L^2(\Lambda)\right)$ is 
a variational solution to Problem \eqref{equation} if it belongs to 
\begin{align*}
 L^2(\Omega;\mathcal{C}([0,T];L^2(\Lambda)))\cap L^2(\Omega;L^2(0,T;H^1(\Lambda)))
\end{align*}
and satisfies, for all $t\in[0,T]$,
\begin{align}\label{240311_eq1}
u(t)-u_0+ \int_0^t \operatorname{div}(\mathbf{v}f(u(s))) - \beta(u(s))+ \Delta u(s)\,ds=\int_0^t g(u(s))\,dW(s)
\end{align}
in $L^2(\Lambda)$ and $\mathds{P}$-a.s. in $\Omega$, where $\Delta$ denotes the Laplace operator on $H^1(\Lambda)$ associated with the weak formulation of the homogeneous Neumann boundary condition.
\end{defn}
\begin{rem}
Existence, uniqueness and regularity of a variational solution to \eqref{equation} in the sense of Definition \ref{solution} is well-known in the literature, see, e.g., \cite{KryRoz81, LR, PardouxThese}.
\end{rem}
\subsection{Admissible meshes and notations}
Let us introduce the temporal and spatial discretizations. For the time-discretization, let $N\in\mathbb{N}$ be given. We define the equidistant time step $\tau=\frac{T}{N}$ and divide the interval $[0,T]$ in $0=t_0<t_1<...<t_N=T$ with $t_n=n \tau$ for all $n\in \{0, ..., N-1\}$.
For the space discretization, we consider admissible meshes in the sense of \cite[Definition 9.1]{EGH00}. For the sake of self-containedness, we add the definition:
\begin{defn}[see \cite{EGH00}, Definition 9.1]\label{defmesh} 
For a bounded domain $\Lambda\subset\mathbb{R}^d$, $d=2$ or $d=3$, an admissible finite-volume mesh $\mathcal{T}$ is given by a family of open, polygonal, and convex subsets $K$, called \textit{control volumes}
and a family of subsets of $\overline{\Lambda}$, contained in hyperplanes of $\mathbb{R}^d$ with strictly positive $(d-1)$-dimensional Lebesgue measure, denoted by $\mathcal{E}$. The elements of $\mathcal{E}$ are called \textit{edges} (for $d=2$) or \textit{sides} (for $d=3$) of the control volumes. Finally we associate a family $\mathcal{P}=(x_K)_{K\in\mathcal{T}}$ of points in $\Lambda$ to the family of control volumes, called \textit{centers}. 
We assume $K\in\mathcal{T}$, $\mathcal{E}$ and $\mathcal{P}$ to satisfy the following properties:
\begin{itemize}
\item $\overline{\Lambda}=\bigcup_{K\in\mathcal{T}}\overline{K}$.
\item For any $K\in\mathcal{T}$, there exists a subset $\mathcal{E}_K$ of $\mathcal{E}$ such that $\partial K=\overline{K}\setminus K=\bigcup_{\sigma\in \mathcal{E}_K}\overline{\sigma}$ and $\mathcal{E}=\bigcup_{K\in\mathcal{T}}\mathcal{E}_K$. $\mathcal{E}_K$ is called the set of edges of $K$ for $d=2$ and sides for $d=3$, respectively.
\item If $K,L\in\mathcal{T}$ with $K\neq L$ then either the $(d-1)$-dimensional Lebesgue measure of $\overline{K}\cap\overline{L}$ is $0$  or $\overline{K}\cap\overline{L}=\overline{\sigma}$ for some $\sigma\in \mathcal{E}$, which will then be denoted by $K|L$.
\item $\mathcal{P}=(x_K)_{K\in\mathcal{T}}$ is such that $x_K\in \overline{K}$ for all $K\in\mathcal{T}$ and, if $K,L\in\mathcal{T}$ are two neighboring control volumes, $x_K\neq x_L$ and the straight line between the centers $x_K$ and $x_L$ is orthogonal to the edge $\sigma=K|L$.
\item For any $\sigma\in\mathcal{E}$ such that $\sigma\subset\partial\Lambda$, let $K$ be the control volume such that $\sigma \in\mathcal{E}_K$. If $x_K\notin\sigma$, the straight line going through $x_K$ and orthogonal to $\sigma$ has a nonempty intersection with $\sigma$.
\end{itemize}
\end{defn}
Examples of admissible meshes are triangular meshes for $d=2$, where the condition that angles of the triangles are less than $\frac{\pi}{2}$ ensures $x_K\in K$. For $d=3$, Voronoi meshes are admissible. See \cite{EGH00}, Example 9.1 and 9.2 for more details.
In the following, we will use the following notation:
\begin{itemize}
\item $h=\operatorname{size}(\mathcal{T})=\sup\{\operatorname{diam}(K): K\in\mathcal{T}\}$ the mesh size.
\item $\mathcal{E}_{\operatorname{int}}:=\{\sigma\in\mathcal{E}:\sigma\nsubseteq \partial\Lambda\}$, $\mathcal{E}_{\operatorname{ext}}:=\{\sigma\in\mathcal{E}:\sigma\subseteq \partial\Lambda\}$, $\mathcal{E}^{K}_{\operatorname{int}}=\mathcal{E}_{\operatorname{int}}\cap \mathcal{E}_{K}$ for $K\in\mathcal{T}$.

\item For $K\in\mathcal{T}$, $\sigma\in\mathcal{E}$ and $d=2$ or $d=3$, let $m_K$ be the $d$-dimensional Lebesgue measure of $K$ and let $m_\sigma$ be the $(d-1)$ dimensional Lebesgue measure of $\sigma$.

\item For $K\in\mathcal{T}$, $\vec{n}_K$ denotes the unit normal vector to $\partial K$ outward to $K$ and for $\sigma \in \mathcal{E}_K$, we denote the unit vector on the edge $\sigma$ pointing out of $K$ by $\vec{n}_{K,\sigma}$.

\item Let $K,L\in\mathcal{T}$ be two neighboring control volumes. For $\sigma=K|L\in\edgesint$, let $d_{K|L}$ denote the Euclidean distance between $x_K$ and $x_L$.
\end{itemize}

Using these notations, we introduce a positive number 
\begin{align}\label{mrp}
\reg=\max\left(\mathcal N,\max_{\genfrac{}{}{0pt}{}{K \in\mathcal{T}}{\sigma\in\mathcal{E}_K}} \frac{\operatorname{diam}(K)}{d(x_K,\sigma)}\right)
\end{align}
(where $\mathcal N$ is the maximum of edges incident to any vertex) that measures the regularity of a given mesh and is useful to perform the convergence analysis of finite-volume schemes.
This number should be uniformly bounded by a constant $\chi>0$ not depending on the mesh size $h$ for the convergence results to hold, \textit{i.e.},
\begin{align}\label{mesh_regularity}
\reg \leq \chi.
\end{align}
We have in particular $\forall K,L\in \mathcal{T}$,  
\begin{align*}
\frac{h}{\dkl}\leq \reg.
\end{align*} 

\subsection{The finite volume scheme}

Let $\mathcal{T}_h$ be an admissible finite volume mesh with mesh size $h>0$. For a $\mathcal{F}_0$-measurable random element $u_0\in L^2(\Omega;L^2(\Lambda))$ we define $\mathbb{P}$-a.s. in $\Omega$ its piecewise constant spatial discretization $u_h^0= \sum_{K \in \mathcal{T}_h} u_K^0 \mathds{1}_K$, where
\begin{align*}
u_K^0= \frac{1}{m_K} \int_K u_0(x) \, dx, ~~\forall K \in \mathcal{T}_h.
\end{align*}
The finite-volume scheme we propose reads, for the initially given $\mathcal{F}_0$-measurable random vector $(u_K^0)_{K\in\mathcal{T}_h}$ induced by $u_h^0$, as follows:
For any $n \in \{1,\cdots,N\}$, given the $\mathcal{F}_{t_{n-1}}$-measurable random vector $(u^{n-1}_K)_{K\in\mathcal{T}_h}$ we search for a $\mathcal{F}_{t_{n}}$-measurable random vector $(u^{n}_K)_{K\in\mathcal{T}_h}$ such that, for almost every $\omega\in\Omega$, $(u^{n}_K)_{K\in\mathcal{T}_h}$ is solution to the following random equations
\begin{align}\label{FVS}
    \begin{aligned}
        &m_K(u_K^n - u_K^{n-1}) +  \tau \sum\limits_{\sigma = K|L \in \mathcal{E}_K^{int}} m_{\sigma} \mathbf{v}_{K, \sigma}^n f(u_{\sigma}^n) \\\
        &+\tau \sum\limits_{\sigma \in \mathcal{E}_K^{int}} \frac{m_{\sigma}}{d_{K|L}} (u_K^n - u_L^n)\\
        &= m_K g(u_K^{n-1}) \Delta_nW + \tau m_K \beta(u_K^n), ~~~\forall K \in \mathcal{T}_h,
    \end{aligned}
\end{align}
where $0=t_0<t_1<...<t_N=T$, $\tau=t_n - t_{n-1}=T/N$, $\Delta W_n:= W(t_n)-W(t_{n-1})$ for all $n \in \{1,...,N\}$,
\begin{align*}
\mathbf{v}_{K, \sigma}^n = \frac{1}{\tau m_{\sigma}}\int_{t_{n-1}}^{t_n} \int_{\sigma} \mathbf{v}(t,x) \cdot \vec{n}_{K, \sigma} \, d\gamma(x) dt,
\end{align*}
where $\gamma$ denotes the $(d-1)$-dimensional Lebesgue measure and $u_{\sigma}^n$ denotes the upstream value at time $t_n$ with respect to $\sigma$ defined as follows: If $\sigma = K|L \in \mathcal{E}_K^{int}$ is the interface between the control volumes $K$ and $L$, $u_{\sigma}^n$ is equal to $u_K^n$ if $\mathbf{v}_{K, \sigma}^n \geq 0$ and to $u_L^n$ if $\mathbf{v}_{K, \sigma}^n <0.$
\begin{rem}
	The second term on the left-hand side of \eqref{FVS} is the classical two-point flux approximation of the Laplace operator (see~\cite[Section 10]{EGH00} for more details on the two-point flux approximation of the Laplace operator with Neumann boundary conditions). 
\end{rem}
With any finite sequence of unknowns $(w_{K})_{K\in\Tau_h}$, we may associate the piecewise constant function
\[w_h(x)=\sum_{K\in \Tau_h} w_K\mathds{1}_K(x), ~x\in\Lambda,\]
where $\mathds{1}_K$ denotes the standard indicator function on $K$, i.e., $\mathds{1}_K(x)=1$ for $x \in K$ and $\mathds{1}_K(x)=0$ for $x \notin K$.
The discrete $L^2$-norm for $w_h$ is then given by
\[\Vert w_h\Vert_2^2=\sum_{K\in\Tau_h}m_K |w_K|^2\]
and we may define its discrete $H^1$-seminorm by
\[
|w_h|_{1,h}^2:= \sum_{\sigma\in\mathcal{E}_{\operatorname{int}}} \frac{m_{\sigma}}{d_{K|L}} (w_K-w_L)^2.\]
Consequently, if $(u^{n}_K)_{K\in\mathcal{T}_h}$ is the solution to \eqref{FVS}, we will consider 
\[u_h^n= \sum_{K \in \mathcal{T}_h} u_K^n \mathds{1}_K\]
for any $n \in \{1,...,N\}$ a.s. in $\Omega$.
\subsection{Technical Lemmas}
In the sequel, the following discrete Gronwall inequality will be frequently used:

\begin{lem}[\cite{S69}, Lemma 1]\label{Gronwall}
Let $N \in \mathbb{N}$ and $a_n, b_n, \alpha \geq 0$ for all $n \in \{1,...,N\}$. Assume that for every $n \in \{1,...,N\}$, 
$$a_n \leq \alpha + \sum\limits_{k=0}^{n-1} a_k b_k.$$
Then, for any $n \in \{1,...,N\}$ we have
\begin{align*}
a_n \leq \alpha \exp\bigg(\sum\limits_{k=0}^{n-1}b_k \bigg).
\end{align*}
\end{lem}
Moreover, we use the following version of discrete Poincar\'{e} inequality:
\begin{lem}[\cite{B-CC-HF15}, Theorem 3.6, see also \cite{Flore21}, Lemma 1]\label{Lemma PI}
There exists a constant $C_p>0$, only depending on $\Lambda$, such that for all admissible meshes $\mathcal{T}_h$ and for all piecewise constant functions $w_h=(w_K)_{K \in \mathcal{T}_h}$ we have
\begin{align}\label{291123_01}
\Vert w_h \Vert_2^2 \leq C_p |w_h|_{{1, h}}^2 + 2|\Lambda|^{-1} \bigg( \int_{\Lambda} w_h(x) \, dx \bigg)^2. \tag{PI}
\end{align}
\end{lem}
We will use the upcoming version of the discrete integration by parts rule.
\begin{rem}[\cite{BNSZ23}, Remark 2.8]\label{DIBP}
Consider two finite sequences $(w_K)_{K \in \Tau_h}, (v_K)_{K \in \Tau_h}$. Then
\begin{align*}
\sum\limits_{K \in \Tau_h} \sum\limits_{\sigma\in\mathcal{E}_K^{\operatorname{int}}} \frac{m_{\sigma}}{d_{K|L}} (w_K - w_L)v_K = \sum\limits_{\sigma\in\mathcal{E}^{\operatorname{int}}} \frac{m_{\sigma}}{d_{K|L}} (w_K - w_L)(v_K - v_L).
\end{align*}
\end{rem}
We will use frequently the following Lemmas already mentioned and proven in \cite{SZ25}. \enlargethispage{1\baselineskip}
\begin{lem}[\cite{SZ25}, Lemma 9.1]\label{Theorem 110124_01}
Let $d\in \{2,3\}$. Then, for any $u \in H^2(\Lambda)$ satisfying the weak homogeneous Neumann boundary condition we have
\begin{align*}
\Vert u \Vert_{H^2(\Lambda)}^2 \leq 12(\Vert \Delta u \Vert_2^2 + \Vert u \Vert_2^2),
\end{align*}
where $\Delta$ denotes the Laplace operator on $H^1(\Lambda)$ associated with the weak formulation of the homogeneous Neumann boundary condition.\\
Especially, for any random variable $u : \Omega \to H^2(\Lambda)$ we obtain
\begin{align*}
\Vert u \Vert_{H^2(\Lambda)}^2 \leq 12(\Vert \Delta u \Vert_2^2 + \Vert u \Vert_2^2) ~~~\mathbb{P}\text{-a.s. in}~\Omega.
\end{align*}
\end{lem}
\begin{lem}[\cite{SZ25}, Lemma 9.2]\label{Lemma 240124_01}
Let $u \in H^2(\Lambda)$, $\mathcal{T}$ an admissible mesh, $K \in \mathcal{T}$ and $y \in K$. Then, for any $ 1 \leq q \leq 2$, there exists a constant $C=C(q, \Lambda)>0$ such that
\begin{align*}
\int_K |u(x) - u(y) |^q \, dx \leq C h^q \Vert u \Vert_{W^{2,q}(K)}^q.
\end{align*}
Especially, for any function $v$ of the form $v(x):= \sum\limits_{K \in \mathcal{T}} u(y_K) \mathds{1}_{K}(x)$, where $y_K \in K$ and $x \in \Lambda$, we have
\begin{align*}
\Vert u - v \Vert_{L^2(\Lambda)}^2 \leq C(2,\Lambda) h^2 \Vert u \Vert_{H^2(\Lambda)}^2.
\end{align*}
\end{lem}

\section{Regularity assumptions on the data and main result}
In order to prove our result, we need the following additional regularity assumptions on the data:
\begin{Assumptions}
    \item \label{A1} $~g \in \mathcal{C}^2(\mathbb{R})$ such that $g'$ and $g''$ are bounded on $\mathbb{R}$. 
    \item \label{A2} $~u_0 \in L^2(\Omega;H^2(\Lambda)) \cap L^{16}(\Omega; H^1(\Lambda))$ is $\mathcal{F}_0$-measurable and satisfies the weak homogeneous Neumann boundary condition.
    \item  \label{A3}  $~f \in \mathcal{C}^2(\mathbb{R})$ non-decreasing such that $f'$ and $f''$ are bounded on $\mathbb{R}$,  $f(0) = 0$.
    \item  \label{A4} $~\beta \in \mathcal{C}^2(\mathbb{R})$ such that $\beta'$ and $\beta''$ are bounded on $\mathbb{R}$,  $\beta(0) = 0$.
    \item  \label{A5} $\mathbf{v} \in C^1([0,T] \times \Lambda; \mathbb{R}^d)$ such that ${\operatorname{div}} (\mathbf{v})=0$ in $[0,T] \times \Lambda$ and $ \mathbf{v} \cdot \vec{n}=0$ on $[0,T] \times \partial \Lambda$.
\end{Assumptions}

\begin{thm}\label{main theorem}
Let \ref{A1} - \ref{A5} be satisfied and let $u$ be the variational solution of the stochastic non-linear parabolic equation \eqref{equation}
in the sense of Definition \ref{solution}. 
For $N \in \mathbb{N}$ and $h>0$, let $\Tau_h$ be an admissible mesh and $(u_h^n)_{n=1,...,N}$ the solution of the finite volume scheme \eqref{FVS}. Then, there exists a constant $\Upsilon>0$ depending on the mesh regularity $\operatorname{reg}(\Tau_h)$ but not depending on $n, N$ and $h$ explicitly such that
\begin{align*}
\sup_{t \in [0,T]} \erww{\Vert u(t) - u_{h,N}^r(t) \Vert_2^2} \leq \Upsilon(\tau + h^2 + \frac{h^2}{\tau}),
\end{align*}
where $u_{h,N}^r(t)=\sum\limits_{n=1}^N u_h^n \mathds{1}_{[t_{n-1}, t_n)}(t)$, $u_{h,N}^r(T)= u_h^N$. If \eqref{mesh_regularity} is satisfied, $\Upsilon$ may depend on $\chi$ and the dependence of $\Upsilon$ on $\operatorname{reg}(\Tau_h)$ can be omitted.
\end{thm}

\begin{rem}
The error estimate of Theorem \ref{main theorem} also applies to  
\begin{align*}
u_{h,N}^l(t)=\sum\limits_{n=1}^N u_h^{n-1} \mathds{1}_{[t_{n-1}, t_n)}(t), ~u_{h,N}(T)=u_h^N.
\end{align*}
instead of $u_{h,N}^r$.
\end{rem}
\begin{rem}
		As already mentioned in \cite{MajPro18}, it is easily possible to generalize the error analysis of our scheme to a cylindrical Wiener process in $L^2(D)$. In this case, we consider a diffusion operator $G(u)$, where $G$ is a Hilbert-Schmidt operator on $L^2(D)$ such that, if $(e_k)_k$ is an orthonormal basis of $L^2(D)$, for any $k\in\mathbb{N}$, $G(u)(e_k)=g_k(u)$, where $g_k\in\mathcal{C}^2(\mathbb{R})$ satisfies
		\[\sum_{k=1}^{\infty}\Vert g'_k\Vert^2_{\infty}<+\infty, \quad  \sum_{k=1}^{\infty}\Vert g''_k\Vert^2_{\infty}<+\infty.\]
		Since the noise is discretized with respect to the time variable, we limit ourselves to the $1$-d case so as not to overload the notation.
\end{rem}
\section{Semi-implicit Euler scheme for a stochastic non-linear parabolic equation}
Now, for a $\mathcal{F}_0$-measurable random variable $v^0\in L^2(\Omega; L^2(\Lambda))$, we consider the semi-implicit Euler scheme
\begin{align}\label{ES}
\begin{cases} v^n - v^{n-1} - \tau \Delta v^n +  \tau  \operatorname{div} (\mathbf{v}^n f(v^{n})) = g(v^{n-1}) \Delta_n W +  \tau \beta(v^{n}) ~&\text{in}~ \Omega \times \Lambda ,\\
\nabla v^n \cdot \vec{n} =0~&\text{on}~\Omega \times \partial \Lambda,
\end{cases} \tag{ES}
\end{align}
where $0=t_0<t_1<...<t_N=T$ and $\tau=t_n - t_{n-1}=\frac{T}{N}$ for all $n \in \{1,...,N\}$.\\
From the Theorem of Stampacchia (see, e.g., \cite[Thm. 5.6]{BrezisFA}) or arguing as in \cite{SWZ19} it follows that there exists a unique $(\mathcal{F}_{t_n})$-measurable solution $v^n$ in $L^2(\Omega;H^1(\Lambda))$ to \eqref{ES} in the weak sense.\\ 
Moreover, from Theorem 3.2.1.3 in \cite{Grisvard85} we obtain $v^n \in L^2(\Omega; H^2(\Lambda))$, hence \eqref{ES} holds a.e. in $\Omega\times \Lambda$ and the homogeneous Neumann boundary condition holds in the weak sense, i.e., for any $w \in H^1(\Lambda)$ we have
\begin{align*}
\int_{\partial \Lambda} w \nabla v^n \cdot \vec{n} \, dS=0,
\end{align*}
where the integrands $\nabla v^n$ and $w$ have to be understood in the sense of trace operators.\\
Furthermore, for any random variable $u : \Omega \to H^2(\Lambda)$ satisfying the weak homogeneous Neumann boundary conditions we have
\begin{align}\label{eq 041023_02}
\Vert u \Vert_{H^2(\Lambda)}^2 \leq 12(\Vert \Delta u \Vert_2^2 + \Vert u \Vert_2^2) ~~~\mathbb{P}\text{-a.s. in}~\Omega,
\end{align}
see Appendix, Lemma \ref{Theorem 110124_01}.
\begin{lem}\label{Lemma 250923_01}
 There exist constants $C_1, K_1 > 0$ and $N_1 \in \N$ such that for any $N \geq N_1$ we have
\begin{align*}
\sup_{n \in \{1,...,N\}} \erww{\Vert v^n \Vert_{L^2}^{2}} + \sum\limits_{n=1}^N \erww{\Vert v^n - v^{n-1} \Vert_{L^2}^{2}}\leq K_1 \erww{\Vert v^0 \Vert_{L^2}^{2}} \leq C_1.
\end{align*}
\end{lem}
\begin{proof}
 We take a test function $v^n$ in  \eqref{ES} to get
 \begin{align*}
      (v^n - v^{n-1},v^n) &- \tau (\Delta v^n, v^n) +  \tau  ({\operatorname{div}} (\mathbf{v}^n f(v^{n})), v^n) \\
      & = (g(v^{n-1})( W(t_n)-W(t_{n-1})), v^n-v^{n-1} + v^{n-1})  +  \tau (\beta(v^{n}), v^n) .
 \end{align*}
Thanks to the assumption $\mathbf{v} \cdot \vec{n} =0$ on $\partial \Lambda \times [0,T]$, by setting $F(r) := \int_0^r f(r') \, dr'$, we have
\begin{align}\label{180925_01}
\begin{aligned}
    \int_{\Lambda} \mathbf{v}^n f(v^n) \cdot \nabla v^n \, dx = \int_{\Lambda} \mathbf{v}^n \cdot \nabla F(v^n) \, dx = \int_{\Lambda} \operatorname{div}(\mathbf{v}^n) F(v^n) \, dx = 0.
\end{aligned}
\end{align}
Now, using Young's inequality, the assumptions on $\mathbf{v}$ and $\beta$,
and the identity 
   \begin{align}
    (x -y) x = \frac{1}{2}\left [x^2 - y^2 + {(x-y)}^2 \right ]~~~\forall \, x,y \in \mathbb{R}, \label{identity}
    \end{align}
  we arrive at 
\begin{align}
    \frac{1}{2} \, \bigg[\Vert v^n \Vert_{L^2}^{2} - \Vert v^{n-1} \Vert_{L^2}^{2} & + \Vert v^n - v^{n-1} \Vert_{L^2}^{2} \bigg] + \tau \left \| \nabla v^n \right \|_{L^2}^{2} \notag \\ 
    &\leq   (g(v^{n-1})( W(t_n)-W(t_{n-1})), v^n-v^{n-1} + v^{n-1}) + \tau \, L_{\beta} \|v^n \|_{L^2}^{2} . \notag 
\end{align}
 
After taking expectation and using Young's inequality yields that

\begin{align}
    &\frac{1}{2}\erww{\Vert v^n \Vert_{L^2}^{2} - \Vert v^{n-1} \Vert_{L^2}^{2} + \Vert v^n - v^{n-1} \Vert_{L^2}^{2}} + \tau \erww{\left | \nabla v^n \right \|_{L^2}^{2}} \notag \\
    &\leq  \tau \erww{\left \| g(v^{n-1})\right \|_{L^2}^{2}} + \frac{1}{4} \erww{\left \| v^n-v^{n-1}\right \|_{L^2}^{2}} + \tau \, L_{\beta} \erww{\|v^n \|_{L^2}^{2}} . \notag
\end{align}
 For fixed $n \in \left \{1,2,\cdots,N  \right \}$, we take the sum over $k=1,2,...,n$ in above inequality and obtain 
\begin{align}
    \frac{1}{2} \, &\erww{\Vert v^n \Vert_{L^2}^{2}}  + \frac{1}{4} \,\sum_{k=1}^{n}   \erww{\Vert v^k - v^{k-1} \Vert_{L^2}^{2}}   + \tau \sum_{k=1}^{n} \erww{\Vert \nabla v^k \Vert_{L^2}^{2}} \notag \\
    & \le  \frac{1}{2} \, \erww{\Vert v^0 \Vert_{L^2}^{2}} + \tau \,L_{g}^2  \sum_{k=1}^{n} \erww{\Vert v^{k-1} \Vert_{L^2}^{2}} + \tau \, \sum_{k=1}^{n} \left \| g(0)\right \|_{L^2}^{2} + \tau \, L_{\beta} \sum_{k=1}^{n}  \erww{\Vert v^k \Vert_{L^2}^{2} } \notag \\
    & \le  \frac{1}{2} \, \erww{\Vert v^0 \Vert_{L^2}^{2} } + \tau \,(L_{g}^2  + L_{\beta} ) \sum_{k=1}^{n}\erww{\Vert v^{k-1} \Vert_{L^2}^{2} } + C_g T |\Lambda| + \tau \, L_{\beta}  \erww{\Vert v^n \Vert_{L^2}^{2} } . \notag
\end{align}
Now, there exists $N_1 \in \N$ such that $\tau L_{\beta} \leq \frac{1}{4}$. Hence, for $N \geq N_1$ and $\tau = \frac{T}{N}$ we have $\frac{1}{2} - \tau L_{\beta} \geq\frac{1}{4}$. Multiplying with $4$ yields
\begin{align}\label{eq:estimation-1}
    \begin{aligned}
        \erww{\Vert v^n \Vert_{L^2}^{2}} & + \sum_{k=1}^{n} \erww{\Vert v^k - v^{k-1} \Vert_{L^2}^{2}}   +4\,  \tau \sum_{k=1}^{n} \erww{\Vert \nabla v^k \Vert_{L^2}^{2}}\\ 
        & \le  2  \, \erww{\Vert v^0 \Vert_{L^2}^{2}} + 4 C_g T + 4 \, \tau \,(L_{g}^2  + L_{\beta} ) \sum_{k=1}^{n} \erww{\Vert v^{k-1} \Vert_{L^2}^{2}} .
    \end{aligned}  
\end{align}
Then, discarding the non-negative terms and applying Lemma \ref{Gronwall} yields the existence of a constant $C=C(T,g, \beta)>0$ such that
\begin{align*}
    \sup\limits_{n=1,...,N} \erww{\Vert v^n \Vert_{L^2}^{2}} \leq C \erww{\Vert v^0 \Vert_{L^2}^{2}}.
\end{align*}
Now, using \eqref{eq:estimation-1}  with $k=N$ yields the assertion.
\end{proof}
%We may remark that the set of all $N \in \N$ with $\frac{1}{2} - \tau \, L_{\beta} \leq 0$ is finite and therefore the assertion of Lemma \ref{Lemma 250923_01} still holds true for all $N \in \N$.
\begin{rem}
Let us remark that if $v^0 \in L^2(\Omega; H^1(\Lambda))$, then since
\begin{equation*}
    \Delta v^n = \tau^{-1} (v^n - v^{n-1} + \operatorname{div}(\mathbf{v}f(v^n)) - g(v^{n-1}) \Delta_nW + \beta(v^n))
\end{equation*}
for each $N \in \N$ and $n \in \{1,...,N\}$, we automatically obtain that $\Delta v^n$ is $(\mathcal{F}_{t_n})$-measurable with values in $H^1(\Lambda)$. Moreover, by Theorem 3.2.1.3 in \cite{Grisvard85} and the fact that $\partial_i v^n$ satisfies the homogeneous Neumann boundary condition in the weak sense for all $i=1,...,d$, we also get that $v^n \in L^2(\Omega; H^3(\Lambda))$ for $n \in \{1,...,N\}$. If $v^0 \in L^2(\Omega; H^2(\Lambda))$, then we obtain automatically $\Delta v^n \in L^2(\Omega; H^2(\Lambda))$.
\end{rem}
\begin{lem}\label{Lemma 250923_02} 
Let $r \in \N_0$ and $v_0 \in L^{2^{r+1}}(\Omega; H^1(\Lambda))$ be $\mathcal{F}_0$-measurable. Then there exist constants $\tilde{K}_2(r), \hat{C}(r)>0$ and $\tilde{N}_r >0$ such that for any $N \geq \tilde{N}_r$ we have 
\begin{align}\label{eq 240312_1}
\begin{aligned}
\sup_{n \in \{1,...,N\}} &\erww{\Vert v^n \Vert_{H^1}^{2^{r+1}}}\\ 
&+ \sum\limits_{n=1}^N \mathbb{E}\left[\prod\limits_{l=1}^r \Big[\Vert v^n \Vert_{H^1}^{2^l} + \Vert v^{n-1} \Vert_{H^1}^{2^l} \Big] \times \Big(\Vert v^n - v^{n-1} \Vert_{H^1}^2 + \tau \Vert \Delta v^n \Vert_2^2 \Big)\right] \\
&\leq \hat{C}(r) \erww{\Vert v^0 \Vert_{H^1}^{2^{r+1}}} \leq \tilde{K}_2(r).
\end{aligned}
\end{align}

\end{lem}
\begin{proof}
The proof is similar to the proof of Lemma 4.1 in \cite{MajPro18}, however, for the sake of completeness, we want to give the whole proof by induction.\\
We firstly claim the following: For any $r \in \N_0$ and any $n=1,\ldots,N$, there exists $K(r)>0$ and a $(\mathcal{F}_{t_{n-1}})$-measurable real-valued random variable $f_{n-1}^r$ with $|f_{n-1}^r|^2 \leq \tilde{C}(r)(1+ \Vert v^{n-1} \Vert_{H^1}^{2^{r+2}})$ for some constant $\tilde{C}(r)$ only depending on $r$ such that
\begin{align}\label{eq 020224_01}
%\begin{aligned}
 & \Big[\Vert v^n \Vert_{H^1}^{2^{r+1}} - \Vert v^{n-1} \Vert_{H^1}^{2^{r+1}}\Big] +  \bigg[\prod\limits_{l=1}^r \Big[\Vert v^n \Vert_{H^1}^{2^l} + \frac{1}{2} \Vert v^{n-1} \Vert_{H^1}^{2^l} \Big] \times \Big( \frac{1}{2} \Vert v^n - v^{n-1} \Vert_{H^1}^2 + \tau \Vert \Delta v^n \Vert_2^2 \Big)\bigg] \notag\\
&\leq C(r) \sum\limits_{l=1}^{r+1}|\Delta_n W|^{2^l} (1+ \Vert v^{n-1} \Vert_{H^1}^{2^{r+1}}) +f_{n-1}^r \Delta_n W \notag\\
&\hspace{3cm} + \left(\frac{8}{3}\right)^r \tau \, C_{f', \mathbf{v}, \beta'} \left (\Vert v^n \Vert_{H^1}^{2^{r+1}} + \frac{1}{2} \Vert v^{n-1} \Vert_{H^1}^{2^{r+1}} \right)
%\end{aligned}
\end{align}
$\mathbb{P}$-a.s. in $\Omega$. We prove \eqref{eq 020224_01} for any $r \in \N_0$ and any $n=1,\ldots,N$ by induction over $r$. First, for $r=0$, we prove
\begin{align*}
 & \Vert v^n \Vert_{H^1}^2 - \Vert v^{n-1} \Vert_{H^1}^2 +  \frac{1}{2} \Vert v^n - v^{n-1} \Vert_{H^1}^2 + \tau \Vert \Delta v^n \Vert_2^2 \\
&\leq C(0) |\Delta_n W|^2 (1+ \Vert v^{n-1} \Vert_{H^1}^2) +f_{n-1}^0 \Delta_n W + \tau \, C_{f', \mathbf{v}, \beta'} \Vert v^n \Vert_{H^1}
\end{align*}
$\mathbb{P}$-a.s. in $\Omega$, where we use the convention $\prod\limits_{l=1}^0=1$. Applying $v^n - \Delta v^n$ to \eqref{ES} yields
\begin{align*}
   \frac{1}{2} &\Big[\Vert v^n \Vert_{H^1}^2 - \Vert v^{n-1} \Vert_{H^1}^2+ \Vert v^n - v^{n-1} \Vert_{H^1}^2 \Big] + \tau \Big( \Vert \nabla v^n \Vert_2^2 + \Vert \Delta v^n \Vert_2^2 \Big) \\
&+ \tau \int_{\Lambda} \operatorname{div} (\mathbf{v}^n f(v^n)) (v^n - \Delta(v^n)) \, dx\\ 
&=(g(v^{n-1}) \Delta_n W,  v^n)_2 + (\nabla g(v^{n-1}) \Delta_n W, \nabla v^n)_2 + \tau (\beta(v^n), v^n)_2 + \tau (\nabla \beta(v^n) , \nabla v^n)_2 \\
&=(g(v^{n-1}) \Delta_n W,  v^n- v^{n-1})_2 + (\nabla g(v^{n-1}) \Delta_n W, \nabla v^n - \nabla v^{n-1})_2 \\
&+(g(v^{n-1}) \Delta_n W,  v^{n-1})_2 + (\nabla g(v^{n-1}) \Delta_n W, \nabla v^{n-1})_2 \\
&+ \tau (\beta(v^n), v^n)_2 + \tau (\nabla \beta(v^n) , \nabla v^n)_2 \\
&\leq \Vert g(v^{n-1}) \Vert_2^2 |\Delta_n W|^2 + \frac{1}{4} \Vert v^n - v^{n-1} \Vert_2^2 + \Vert \nabla g(v^{n-1}) \Vert_2^2 |\Delta_n W|^2 + \frac{1}{4} \Vert \nabla v^n - \nabla v^{n-1} \Vert_2^2 \\
\end{align*}
\begin{align*}
&+ \bigg( \int_{\Lambda} g(v^{n-1}) v^{n-1} + g'(v^{n-1}) |\nabla v^{n-1}|^2 \, dx \bigg) \Delta_n W  + \tau \| \beta' \|_{\infty} \| v^n \|^2_{H^1}\\
&\leq \frac{1}{2}C(0) |\Delta_n W|^2 (1+ \Vert v^{n-1} \Vert_{H^1}^2) + \frac{1}{4} \Vert v^n - v^{n-1} \Vert_{H^1}^2 + \frac{1}{2} f_{n-1}^0 \Delta_n W + \tau \| \beta' \|_{\infty} \| v^n \|^2_{H^1}
\end{align*}
for some constant $C(0)>0$ only depending on $C_g,\Vert g' \Vert_{\infty}$, $\Lambda$ and
\begin{equation*}
    f_{n-1}^0 := 2\left( \int_{\Lambda} g(v^{n-1}) v^{n-1} + g'(v^{n-1}) |\nabla v^{n-1}|^2 \, dx \right).
\end{equation*}
Moreover, thanks to \eqref{180925_01} and \ref{A5}, we have
\begin{align*}
    &\int_{\Lambda} \operatorname{div} (\mathbf{v}^n f(v^n)) (v^n - \Delta(v^n)) \, dx = - \int_{\Lambda} \operatorname{div} (\mathbf{v}^n f(v^n))  \Delta(v^n) \, dx \\
    &= \int_{\partial \Lambda} \mathbf{v}^n f(v^n))  \Delta v^n \cdot \vec{n}\, dS+ \int_{\Lambda} \mathbf{v}^n f(v^n)  \nabla \Delta(v^n) \, dx \\
    &= \int_{\Lambda} \mathbf{v}^n f(v^n)  \nabla \Delta(v^n) \, dx = \int_{\Lambda} \operatorname{div}(\mathbf{v}^n f(v^n))  \Delta v^n \, dx \\
    &= \int_{\Lambda} \mathbf{v}^n \cdot\nabla f(v^n))  \Delta v^n \, dx 
\end{align*}
and therefore
\begin{align*}
    &\left| \tau \int_{\Lambda} \operatorname{div} (\mathbf{v}^n f(v^n)) (v^n - \Delta(v^n)) \, dx \right| \leq \frac{\tau}{2} \left( \| \mathbf{v}^n \cdot \nabla f(v^n) \|_{L^2}^2 + \| \Delta v^n \|_{L^2}^2 \right) \\
    &\leq \frac{\tau}{2} \|\mathbf{v} \|_{\infty}^2 \| f'(v^n) \nabla v^n \|_{L^2}^2 + \frac{\tau}{2} \| \Delta v^n \|_{L^2}^2 \\
    &\leq \frac{\tau}{2} \|\mathbf{v} \|_{\infty}^2 \| f' \|_{\infty}^2 \| \nabla v^n \|_{L^2}^2 + \frac{\tau}{2} \| \Delta v^n \|_{L^2}^2.
\end{align*}
Subtracting $\frac{1}{4} \Vert v^n - v^{n-1} \Vert_{H^1}^2$ and $\frac{\tau}{2} \| \Delta v^n \|_{L^2}^2$, multiplying by $2$ and discarding nonnegative terms on the left-hand side yields
\begin{align*}
    & \Vert v^n \Vert_{H^1}^2 - \Vert v^{n-1} \Vert_{H^1}^2 +  \frac{1}{2} \Vert v^n - v^{n-1} \Vert_{H^1}^2 + \tau \Vert \Delta v^n \Vert_2^2 \\
&\leq C(0) |\Delta_n W|^2 (1+ \Vert v^{n-1} \Vert_{H^1}^2) +f_{n-1}^0 \Delta_n W + \tau \left( \|\mathbf{v} \|_{\infty}^2 \| f' \|_{\infty}^2 \| + 2 \| \beta' \|_{\infty} \right) \Vert v^n \Vert^2_{H^1}
\end{align*}
and therefore \eqref{eq 020224_01} holds for $r=0$.\\
Now, we assume \eqref{eq 020224_01} holds true for some $r \in \N_0$ and we show that \eqref{eq 020224_01} holds true for $r+1$. Multiplying \eqref{eq 020224_01} with $(\Vert v^n \Vert_{H^1}^{2^{r+1}} + \frac{1}{2} \Vert v^{n-1} \Vert_{H^1}^{2^{r+1}} )$ yields
\begin{align*}
  &\frac{3}{4} \bigg(\Vert  v^n \Vert_{H^1}^{2^{r+2}} - \Vert v^{n-1} \Vert_{H^1}^{2^{r+2}}\bigg) + \frac{1}{4} \bigg(\Vert  v^n \Vert_{H^1}^{2^{r+1}} - \Vert  v^{n-1} \Vert_{H^1}^{2^{r+1}}\bigg)^2 \\  
  & \quad+  \bigg[\prod\limits_{l=1}^{r+1} \Big[\Vert v^n \Vert_{H^1}^{2^l} + \frac{1}{2} \Vert v^{n-1} \Vert_{H^1}^{2^l} \Big] \times \Big( \frac{1}{2} \Vert v^n - v^{n-1} \Vert_{H^1}^2 + \tau \Vert \Delta v^n \Vert_2^2 \Big)\bigg] \notag 
\end{align*}
\begin{align*}
&\leq \bigg(C(r) \sum\limits_{l=1}^{r+1}|\Delta_n W|^{2^l} (1+ \Vert v^{n-1} \Vert_{H^1}^{2^{r+1}})\bigg) (\Vert v^n \Vert_{H^1}^{2^{r+1}} + \frac{1}{2} \Vert v^{n-1} \Vert_{H^1}^{2^{r+1}} ) \\
&\quad +\bigg( f_{n-1}^r \Delta_n W\bigg) (\Vert v^n \Vert_{H^1}^{2^{r+1}} + \frac{1}{2} \Vert v^{n-1} \Vert_{H^1}^{2^{r+1}} ) + \left( \frac{8}{3}\right)^r \tau \, C_{f', \mathbf{v}, \beta'} (\Vert v^n \Vert_{H^1}^{2^{r+1}} + \frac{1}{2} \Vert v^{n-1} \Vert_{H^1}^{2^{r+1}} )^2\\
&\leq 2C(r)^2 \bigg(\sum\limits_{l=1}^{r+1}|\Delta_n W|^{2^l}\bigg)^2 (1+ \Vert v^{n-1} \Vert_{H^1}^{2^{r+1}})^2 + \frac{1}{8} \bigg(\Vert  v^n \Vert_{H^1}^{2^{r+1}} - \Vert  v^{n-1} \Vert_{H^1}^{2^{r+1}}\bigg)^2 \\
& \quad+ \frac{3}{2} C(r) \bigg(\sum\limits_{l=1}^{r+1}|\Delta_n W|^{2^l}\bigg)(1+ \Vert v^{n-1} \Vert_{H^1}^{2^{r+1}})\Vert v^{n-1} \Vert_{H^1}^{2^{r+1}} + 2|f_{n-1}^r|^2 |\Delta_n W|^2  \\
&\quad + \frac{1}{8} \bigg(\Vert  v^n \Vert_{H^1}^{2^{r+1}} - \Vert  v^{n-1} \Vert_{H^1}^{2^{r+1}}\bigg)^2 + \frac{3}{2} f_{n-1}^r \Delta_n W \Vert v^{n-1} \Vert_{H^1}^{2^{r+1}} \\
&\quad + 2 \cdot \left( \frac{8}{3}\right)^r \tau \, C_{f', \mathbf{v}, \beta'} (\Vert v^n \Vert_{H^1}^{2^{r+2}} + \frac{1}{2} \Vert v^{n-1} \Vert_{H^1}^{2^{r+2}} )\\
&\leq 4C(r)^2(r+1)^2 \bigg(\sum\limits_{l=1}^{r+1}|\Delta_n W|^{2^{l+1}}\bigg)(1+ \Vert v^{n-1} \Vert_{H^1}^{2^{r+2}}) +  \frac{1}{4} \bigg(\Vert  v^n \Vert_{H^1}^{2^{r+1}} - \Vert  v^{n-1} \Vert_{H^1}^{2^{r+1}}\bigg)^2 \\
&\quad + 3C(r) \bigg(\sum\limits_{l=1}^{r+1}|\Delta_n W|^{2^l}\bigg)(1+ \Vert v^{n-1} \Vert_{H^1}^{2^{r+2}}) + 2\tilde{C}(r)(1+ \Vert v^{n-1} \Vert_{H^1}^{2^{r+2}})|\Delta_n W|^2  \\
&\quad + \frac{3}{2} f_{n-1}^r \Vert  v^{n-1} \Vert_{H^1}^{2^{r+1}} \Delta_n W + 2 \cdot \left( \frac{8}{3}\right)^r \tau \, C_{f', \mathbf{v}, \beta'} (\Vert v^n \Vert_{H^1}^{2^{r+2}} + \frac{1}{2} \Vert v^{n-1} \Vert_{H^1}^{2^{r+2}} )\\
&\leq \big(4(r+1)^2C(r)^2 + 3C(r) + 2\tilde{C}(r)\big)\bigg(\sum\limits_{l=1}^{r+2}|\Delta_n W|^{2^l}\bigg)(1+ \Vert v^{n-1} \Vert_{H^1}^{2^{r+2}}) \\
&\quad + \frac{3}{4} f_{n-1}^{r+1} \Delta_n W + \frac{1}{4} \bigg(\Vert  v^n \Vert_{H^1}^{2^{r+1}} - \Vert  v^{n-1} \Vert_{H^1}^{2^{r+1}}\bigg)^2 \\
&\quad + 2 \cdot \left( \frac{8}{3}\right)^r \tau \, C_{f', \mathbf{v}, \beta'} (\Vert v^n \Vert_{H^1}^{2^{r+2}} + \frac{1}{2} \Vert v^{n-1} \Vert_{H^1}^{2^{r+2}} )
\end{align*}
$\mathbb{P}$-a.s. in $\Omega$, where $f_{n-1}^{r+1}:= 2f_{n-1}^r \Vert v^{n-1} \Vert_{H^1}^{2^{r+1}}$. 
Hence, subtracting $\frac{1}{4} \bigg(\Vert  v^n \Vert_{H^1}^{2^{r+1}} - \Vert  v^{n-1} \Vert_{H^1}^{2^{r+1}}\bigg)^2$ yields
\begin{align*}
&\frac{3}{4} \bigg(\Vert  v^n \Vert_{H^1}^{2^{r+2}} - \Vert v^{n-1} \Vert_{H^1}^{2^{r+2}}\bigg) +  \bigg[\prod\limits_{l=1}^{r+1} \Big[\Vert v^n \Vert_{H^1}^{2^l} + \frac{1}{2} \Vert v^{n-1} \Vert_{H^1}^{2^l} \Big] \times \Big( \frac{1}{2} \Vert v^n - v^{n-1} \Vert_{H^1}^2 + \tau \Vert \Delta v^n \Vert_2^2 \Big)\bigg] \notag \\
&\leq \frac{3}{4} C(r+1) \bigg(\sum\limits_{l=1}^{r+2}|\Delta_n W|^{2^l}\bigg)(1+ \Vert v^{n-1} \Vert_{H^1}^{2^{r+2}}) + \frac{3}{4} f_{n-1}^{r+1} \Delta_n W \\
&\quad + 2 \cdot \left( \frac{8}{3}\right)^r \tau \, C_{f', \mathbf{v}, \beta'} (\Vert v^n \Vert_{H^1}^{2^{r+2}} + \frac{1}{2} \Vert v^{n-1} \Vert_{H^1}^{2^{r+2}} )
\end{align*}
$\mathbb{P}$-a.s. in $\Omega$, where $C(r+1):= \frac{4}{3} \big(4(r+1)C(r)^2 + 3C(r) + 2\tilde{C}(r)\big)$. Multiplying with $\frac{4}{3}$ and discarding nonnegative terms on the left-hand side yields \eqref{eq 020224_01} for $r+1$.\\
Now, we have $\erww{\sum\limits_{l=1}^{r+1}|\Delta_n W|^{2^l}}= \sum\limits_{l=1}^{r+1} \tau^{2^{l-1}} (2^l-1)!! \leq \tilde{C}(T, r)\, \tau$ for some constant $\tilde{C}(T, r)>0$, where
\begin{align*}
k!!=
\begin{cases}
k \cdot (k-2) \cdot ... \cdot 2, &k \in \N ~\text{even}, \\
k \cdot (k-2) \cdot ... \cdot 1, &k \in \N ~\text{odd}
\end{cases}
\end{align*}
and $\erww{f_{n-1}^r \Delta_n W} = \erww{f_{n-1}^r}\erww{\Delta_n W}=0$. Therefore, applying the expectation in \eqref{eq 020224_01} and setting $\overline{C}(T,r):= C(r) \tilde{C}(T,r)$ yields
\begin{align*}
&\erww{\Vert v^n \Vert_{H^1}^{2^{r+1}}} - \erww{\Vert v^{n-1} \Vert_{H^1}^{2^{r+1}}} \\
& \quad+ \erww{\prod\limits_{l=1}^r \Big[\Vert v^n \Vert_{H^1}^{2^l} + \frac{1}{2} \Vert v^{n-1} \Vert_{H^1}^{2^l} \Big] \times \Big( \frac{1}{2} \Vert v^n - v^{n-1} \Vert_{H^1}^2 + \tau \Vert \Delta v^n \Vert_2^2 \Big)} \\
&\leq \overline{C}(T,r) \tau + \overline{C}(T,r) \tau \E \Vert v^{n-1} \Vert_{H^1}^{2^{r+1}} + \left( \frac{8}{3}\right)^r \tau \, C_{f', \mathbf{v}, \beta'} \erww{\Vert v^n \Vert_{H^1}^{2^{r+2}} + \frac{1}{2} \Vert v^{n-1} \Vert_{H^1}^{2^{r+2}} }
\end{align*}
for all $r \in \N_0$. Exchanging $n$ and $k$ and summing over $k=1,\ldots, n$ for arbitrary $n\in \{1,\ldots,N\}$ yields
\begin{align}\label{eq 241125_01}
        &\erww{\Vert v^n \Vert_{H^1}^{2^{r+1}}}+ \sum\limits_{k=1}^n \erww{\prod\limits_{l=1}^r \Big[\Vert v^k \Vert_{H^1}^{2^l} + \frac{1}{2} \Vert v^{k-1} \Vert_{H^1}^{2^l} \Big] \times \Big( \frac{1}{2} \Vert v^k - v^{k-1} \Vert_{H^1}^2 + \tau \Vert \Delta v^k \Vert_2^2 \Big)} \notag \\
        &\leq \erww{ \|v^0 \|_{H^1}^{2^{r+1}}} + \tau \, \overline{C}_1(T,r) + \tau \, \overline{C}_2(T,r, f', \mathbf{v}, \beta') \sum\limits_{k=1}^n \erww{\Vert v^{k-1} \Vert_{H^1}^{2^{r+1}}} \notag \\
        & \quad + \left( \frac{8}{3}\right)^r \tau \, C_{f', \mathbf{v}, \beta'} \erww{\Vert v^n \Vert_{H^1}^{2^{r+1}}}. 
\end{align}
Now, choosing $\tilde{N}_r$ such that $\left(\frac{8}{3}\right)^r \tau \, C_{f', \mathbf{v}, \beta'} < \frac{1}{2}$ for all $N \geq \tilde{N}_r$, applying Lemma \ref{Gronwall} yields the existence of a constant $C=C(T,r,f',\mathbf{v},\beta')>0$ such that for all $N \geq \tilde{N}_r$
\begin{align*}
    \sup\limits_{n=1,...,N} \erww{\Vert v^n \Vert_{H^1}^{2^{r+1}}} \leq C \erww{\Vert v^0 \Vert_{H^1}^{2^{r+1}}}.
\end{align*}
Using again \eqref{eq 241125_01} yields the result.
\end{proof}
\begin{cor}\label{Corollary 241125_01}
Let $v^0 \in L^{16}(\Omega; H^1(\Lambda))$. Then there exist constants $C_2, k_2, K_2>0$ and $N_2 \in \N$ such that for any $N \geq N_2$
\begin{align*}
\sup_{n \in \{1,...,N\}} \erww{\Vert v^n \Vert_2^4 } &+ \sup_{n \in \{1,...,N\}} \erww{ \Vert \nabla v^n \Vert_2^{10}} + \tau \sum\limits_{n=1}^N \erww{ \Vert \Delta v^n \Vert_2^2 \Vert \nabla v^n \Vert_2^2} \\
& \leq k_2+ K_2 \erww{\Vert v^0 \Vert_{H^1}^{16}} \leq C_2.
\end{align*}
\end{cor}
\begin{proof}
For $r=1$ in Lemma \ref{Lemma 250923_02}, we obtain $\tilde{N}_1 \in \N$, $\hat{C}(1)>0$ and $K(1)>0$ s.t.
\begin{align*}
    \sup_{n \in \{1,...,N\}} \erww{\Vert v^n \Vert_2^4 } & + \tau \sum\limits_{n=1}^N \erww{ \Vert \Delta v^n \Vert_2^2 \Vert \nabla v^n \Vert_2^2} \\
     &\leq \hat{C}(1) \erww{\Vert v^0 \Vert_{H^1}^4 }\leq \hat{C}(1) \left(1+ \erww{\Vert v^0 \Vert_{H^1}^{16} }\right)\leq K(1)
\end{align*}
for all $N \geq \tilde{N}_1$. For $r=3$ in Lemma \ref{Lemma 250923_02}, we obtain $\tilde{N}_3 \in \N$, $\hat{C}(3) >0$ and $K(3)>0$ s.t.
\begin{align*}
    \sup_{n \in \{1,...,N\}} \erww{\Vert \nabla v^n \Vert_2^{10}} \leq 1 + \sup_{n \in \{1,...,N\}} \erww{\Vert \nabla v^n \Vert_2^{16}} \leq 1 + \hat{C}(3) \erww{\Vert v^0 \Vert_{H^1}^{16}} \leq K(3)
\end{align*}
for all $N \geq \tilde{N}_3$. Setting $k_2 := 1 + \hat{C}(1)$, $N_2 := \max(\tilde{N}_1, \tilde{N}_3)$, $K_2 := \hat{C}_1 + \hat{C}(3)$ and $C_2 :=K(1) + K(3)$, we have
\begin{align*}
    \sup_{n \in \{1,...,N\}} \erww{\Vert v^n \Vert_2^4 } &+ \sup_{n \in \{1,...,N\}} \erww{ \Vert \nabla v^n \Vert_2^{10}} + \tau \sum\limits_{n=1}^N \erww{ \Vert \Delta v^n \Vert_2^2 \Vert \nabla v^n \Vert_2^2}\\
    & \leq k_2+ K_2 \erww{\Vert v^0 \Vert_{H^1}^{16}} \leq C_2.
\end{align*}
\end{proof}
\begin{lem}\label{Lemma 250923_03}
Let $v^0 \in L^2(\Omega; H^2(\Lambda))\cap L^{16}(\Omega; H^1(\Lambda))$ be $\mathcal{F}_0$-measurable with $\Delta v^0 \in L^2(\Omega; H^1(\Lambda))$. Then there exist constants $K_3, C_3>0$ and $N_3 \in \N$ such that for any $N \geq N_3$ we have
\begin{align*}
\sup_{n \in \{1,...,N\}} \erww{\Vert \Delta v^n \Vert_2^2} &+ \sum\limits_{n=1}^N \erww{\Vert \Delta(v^n - v^{n-1}) \Vert_2^2} + \tau \sum\limits_{n=1}^N \erww{\Vert \nabla \Delta v^n \Vert_2^2} \\
&\leq K_3 \left(1 + \erww{\Vert \Delta v^0 \Vert_2^2 +\tau \Vert \nabla \Delta v^0 \Vert_2^2} \right) \leq C_3.
\end{align*}
\end{lem}
\begin{proof}
Applying $\nabla$ to \eqref{ES} and testing with $-\nabla \Delta v^n$ yields
\begin{align}\label{201125_05}
        &\frac{1}{2} \bigg( \Vert \Delta v^n \Vert_2^2 - \Vert \Delta v^{n-1} \Vert_2^2 + \Vert \Delta v^n - \Delta v^{n-1} \Vert_2^2 \bigg) + \tau \Vert \nabla \Delta v^n \Vert_2^2 - \tau\int_{\Lambda} \nabla \operatorname{div}(\mathbf{v}^nf(v^n)) \nabla \Delta v^n \, dx \notag \\
&= (\Delta g(v^{n-1}) \Delta_n W, \Delta v^n)_2 - \tau (\nabla \beta(v^n), \nabla \Delta v^n)_2.
\end{align}
Then
\begin{align}\label{201125_04}
    |\tau (\nabla \beta(v^n), \nabla \Delta v^n)_2| \leq 2\tau \Vert \beta'\Vert_{\infty} \Vert \nabla v^n \Vert_{L^2}^2 + \frac{\tau}{8} \Vert \nabla \Delta v^n \Vert_2^2.
\end{align}
Moreover, we have
\begin{align*}
    \nabla \operatorname{div}(\mathbf{v}^nf(v^n)) &= \nabla(\mathbf{v}^n \nabla f(v^n)) = \nabla \mathbf{v}^n \nabla f(v^n) + \mathbf{v}^n \nabla^2 f(v^n) \\
    &= \nabla \mathbf{v}^n (f'(v^n) \nabla v^n) + f''(v^n) \nabla v^n (\nabla v^n)^T \cdot \mathbf{v}^n + f'(v^n) \nabla^2 v^n \mathbf{v}^n
\end{align*}
and therefore, by Lemma \ref{Theorem 110124_01} there exists a constant $C>0$ depending on $\mathbf{v}$ and $f$ s.t.
\begin{align}\label{201125_01}
        &\left| \tau\int_{\Lambda} \nabla \operatorname{div}(\mathbf{v}^nf(v^n)) \nabla \Delta v^n \, dx \right| \leq 2\tau \Vert \nabla \operatorname{div}(\mathbf{v}^n f(v^n))\Vert_{L^2}^2 + \frac{\tau}{8} \Vert \nabla \Delta v^n \Vert_{L^2}^2 \notag \\
        &\leq \frac{\tau}{8} \Vert \nabla \Delta v^n \Vert_{L^2}^2 + 6\tau \int_{\Lambda} |\nabla \mathbf{v}^n|^2 |f'(v^n)|^2 |\nabla v^n| + |f''(v^n)|^2 |\nabla v^n|^4 |\mathbf{v}^n|^2 \notag \\
        & \hspace{3cm} + |f'(v^n)|^2 |\nabla^2 v^n|^2 |\mathbf{v}^n|^2 \, dx \notag\\
        &\leq \frac{\tau}{8} \Vert \nabla \Delta v^n \Vert_{L^2}^2 + 6\tau C \left[ \Vert \nabla v^n \Vert_{L^2}^2 + \Vert \nabla v^n \Vert_{L^4}^4 + \Vert \nabla^2 v^n \Vert_{L^2}^2 \right] \notag \\
        &\leq \frac{\tau}{8} \Vert \nabla \Delta v^n \Vert_{L^2}^2 + 6\tau C \left[ \Vert \nabla v^n \Vert_{L^4}^4 + \Vert v^n \Vert_{H^2}^2 \right] \notag\\
        &\leq \frac{\tau}{8} \Vert \nabla \Delta v^n \Vert_{L^2}^2 + 6\tau C \left[ \Vert \nabla v^n \Vert_{L^4}^4 + 12 \Vert v^n \Vert_{L^2}^2 + 12 \Vert \Delta v^n \Vert_{L^2}^2\right].
\end{align}
Thanks to Gagliardo-Nirenberg's inequality and Young's inequality, there exist constants $\tilde{C}_1,\tilde{C}_2,\tilde{C}_3,\tilde{C}_4, \tilde{C}_5>0$ such that
\begin{align}\label{201125_02}
    \begin{aligned}
        \Vert \nabla v^n \Vert_{L^4}^4 &\leq
        \begin{cases} \tilde{C}_1 \Vert \Delta v^n \Vert_2^2 \Vert \nabla v^n \Vert_2^2 + \tilde{C}_2 \Vert \nabla v^n \Vert_2^4 + \tilde{C}_3 \Vert v^n \Vert_2^4, &d=2, \\
         \tilde{C}_1 \Vert \nabla \Delta v^n \Vert_2^{3/2} \Vert \nabla v^n \Vert_2^{5/2} + \tilde{C}_2 \Vert \nabla v^n \Vert_2^4, &d=3.
        \end{cases} \\
        &\leq 
        \begin{cases} \tilde{C}_1 \Vert \Delta v^n \Vert_2^2 \Vert \nabla v^n \Vert_2^2 + \tilde{C}_2 \Vert \nabla v^n \Vert_2^4 + \tilde{C}_3 \Vert v^n \Vert_2^4, &d=2, \\
         \frac{1}{48C_1} \Vert \nabla \Delta v^n \Vert_2^2 + \tilde{C}_4 \Vert \nabla v^n \Vert_2^{10} + \tilde{C}_2 \Vert \nabla v^n \Vert_2^4, &d=3.
        \end{cases} \\
        &\leq \tilde{C}_5 \left[ \Vert \Delta v^n \Vert_2^2 \Vert \nabla v^n \Vert_2^2 + \Vert \nabla v^n \Vert_2^4 + \Vert v^n \Vert_2^4 + \Vert \nabla v^n \Vert_2^{10} \right] + \frac{1}{48C_1} \Vert \nabla \Delta v^n \Vert_2^2.
    \end{aligned}
\end{align}
Hence, for $d \in \{2,3\}$, combining \eqref{201125_01} and \eqref{201125_02} and using the inequality $x^q \leq 1 + x^p$ for $0<q<p$, there exists a constant $\tilde{C}_6>0$ such that
\begin{align}\label{201125_03}
    \begin{aligned}
        &\left| \tau\int_{\Lambda} \nabla \operatorname{div}(\mathbf{v}^nf(v^n)) \nabla \Delta v^n \, dx \right| \\
        &\leq \frac{\tau}{4} \Vert \nabla \Delta v^n \Vert_2^2 + \tau \tilde{C}_6 \left[ \Vert \Delta v^n \Vert_2^2 \Vert \nabla v^n \Vert_2^2 + \Vert \nabla v^n \Vert_2^{10} + \Vert v^n \Vert_2^4 + \Vert \Delta v^n \Vert_2^2 + 1\right].
    \end{aligned}
\end{align}
Following the arguments in \cite{SZ25}, for a constant $\tilde{C}_7$, we get
\begin{align}\label{201125_06}
    \begin{aligned}
        &\erww{(\Delta g(v^{n-1}) \Delta_n W, \Delta v^n)_2} \leq \frac{1}{4} \erww{\Vert \Delta v^n - \Delta v^{n-1} \Vert_2^2} + \tau \erww{\Vert \Delta g(v^{n-1})\Vert_2^2} \\
        &\leq \frac{1}{4} \erww{\Vert \Delta v^n - \Delta v^{n-1} \Vert_2^2}  + \frac{\tau}{2} \erww{\Vert \nabla \Delta v^{n-1} \Vert_2^2} \\
        &+ \tau \tilde{C}_7 \erww{\Vert \Delta v^{n-1} \Vert_2^2 + \Vert \Delta v^{n-1} \Vert_2^2 \Vert \nabla v^{n-1} \Vert_2^2 + \Vert \nabla v^{n-1} \Vert_2^{10} + \Vert v^{n-1} \Vert_2^4}.
    \end{aligned}
\end{align}
Taking the expectation in \eqref{201125_05} and using \eqref{201125_04}, \eqref{201125_03} and \eqref{201125_06} yields the existence of a constant $\tilde{C}_8>0$ such that
\begin{align*}
    &\frac{1}{2} \erww{\Vert \Delta v^n \Vert_2^2} - \frac{1}{2} \erww{\Vert \Delta v^{n-1} \Vert_2^2} + \frac{1}{4} \erww{\Vert \Delta v^n - \Delta v^{n-1} \Vert_2^2} + \frac{5\tau}{8} \erww{\Vert \nabla \Delta v^n \Vert_2^2} \\
    &\leq \tau \tilde{C}_8 \erww{\Vert \Delta v^n \Vert_2^2 \Vert \nabla v^n \Vert_2^2 + \Vert \nabla v^n \Vert_2^{10} + \Vert v^n \Vert_2^4 + 1} + \tau \tilde{C}_6 \Vert \Delta v^n \Vert_2^2 \\
    &+ \tau \tilde{C}_7 \erww{\Vert \Delta v^{n-1} \Vert_2^2 + \Vert \Delta v^{n-1} \Vert_2^2 \Vert \nabla v^{n-1} \Vert_2^2 + \Vert \nabla v^{n-1} \Vert_2^{10} + \Vert v^{n-1} \Vert_2^4} + \frac{\tau}{2} \erww{\Vert \nabla \Delta v^{n-1} \Vert_2^2}.
\end{align*}
Exchanging $n$ and $m$ and summing $m$ from $1$ to $n$ yields:
\begin{align*}
    &\frac{1}{2} \erww{\Vert \Delta v^n \Vert_2^2} + \frac{1}{4} \sum\limits_{m=1}^n \erww{\Vert \Delta v^m - \Delta v^{m-1} \Vert_2^2} + \frac{5\tau}{8} \sum\limits_{m=1}^n  \erww{\Vert \nabla \Delta v^m \Vert_2^2} \\
    &\leq \frac{1}{2}\erww{\Vert \Delta v^0 \Vert_2^2} + \tau \tilde{C}_8 \sum\limits_{m=1}^n \erww{\Vert \Delta v^m \Vert_2^2 \Vert \nabla v^m \Vert_2^2 + \Vert \nabla v^m \Vert_2^{10} + \Vert v^m \Vert_2^4 + 1} \\
    & + \tau \tilde{C}_6 \sum\limits_{m=1}^n \erww{\Vert \Delta v^m \Vert_2^2} + \frac{\tau}{2} \sum\limits_{m=1}^n \erww{\Vert \nabla \Delta v^{m-1} \Vert_2^2} \\
    & + \tau \tilde{C}_7 \sum\limits_{m=1}^n \erww{\Vert \Delta v^{m-1} \Vert_2^2 + \Vert \Delta v^{m-1} \Vert_2^2 \Vert \nabla v^{m-1} \Vert_2^2 + \Vert \nabla v^{m-1} \Vert_2^{10} + \Vert v^{m-1} \Vert_2^4}  .
\end{align*}
Hence by Corollary \ref{Corollary 241125_01} we get
\begin{align*}
    &\left(\frac{1}{2} - \tau \tilde{C}_6 \right) \erww{\Vert \Delta v^n \Vert_2^2} + \frac{1}{4} \sum\limits_{m=1}^n \erww{\Vert \Delta v^m - \Delta v^{m-1} \Vert_2^2} + \frac{\tau}{8} \sum\limits_{m=1}^n  \erww{\Vert \nabla \Delta v^m \Vert_2^2} \\
    &\leq \frac{1}{2}\erww{\Vert \Delta v^0 \Vert_2^2} + \tilde{C}_8(C_2 + T) + \tilde{C}_7 C_2(1+T) + \frac{\tau}{2} \erww{\Vert \nabla \Delta v^0 \Vert_2^2} \\
    &+  \tau (\tilde{C}_6 + \tilde{C}_7) \sum\limits_{m=1}^n \erww{\Vert \Delta v^{m-1} \Vert_2^2}.
\end{align*}
Now we choose $N_3 \in \N$ such that for all $N \geq N_3$ we have $1/2 - \tau \tilde{C}_6 \geq 1/4$. Then, by setting $\tilde{C}_9 := 4(\tilde{C}_8(C_2 + T) + \tilde{C}_7 C_2(1+T))$ and $\tilde{C}_{10}:= 4(\tilde{C}_6 + \tilde{C}_7) $
\begin{align}\label{eq 241125_02}
    \begin{aligned}
        &\erww{\Vert \Delta v^n \Vert_2^2} + \sum\limits_{m=1}^n \erww{\Vert \Delta v^m - \Delta v^{m-1} \Vert_2^2} + \frac{\tau}{2} \sum\limits_{m=1}^n  \erww{\Vert \nabla \Delta v^m \Vert_2^2} \\
        &\leq 2\erww{\Vert \Delta v^0 \Vert_2^2} + \tilde{C}_9+ 2\tau \erww{\Vert \nabla \Delta v^0 \Vert_2^2} +  \tilde{C}_{10}\tau \sum\limits_{m=1}^n \erww{\Vert \Delta v^{m-1} \Vert_2^2}.
    \end{aligned}
\end{align}
Applying Lemma \ref{Gronwall} we may conclude the existence of a constant $C$ such that
\begin{align*}
\sup_{n \in \{1,...,N\}} \erww{\Vert \Delta v^n \Vert_2^2} \leq C \left(1 + \erww{\Vert \Delta v^0 \Vert_2^2 +\tau \Vert \nabla \Delta v^0 \Vert_2^2}\right).
\end{align*}
Now, using this estimate in \eqref{eq 241125_02} with $n=N$ yields the assertion.
\end{proof}
\section{Stability estimates for a stochastic non-linear parabolic equation}
\subsection{A regularity result}
\begin{prop}\label{240312_prop1}
Let \ref{A1} to \ref{A5} be satisfied. Then, the unique variational solution $u$ to \eqref{equation} has the additional regularity $u\in L^2(\Omega;\mathcal{C}([0,T];H^2(\Lambda)))$ and satisfies the weak homogeneous Neumann boundary condition.
\end{prop}
\begin{rem}
Additionally we make sure that $\Delta u(t)$, the Neumann-Laplacian of $u$ at time $t \in [0,T]$, is an element of $L^2(\Lambda)$. Actually, this is the case if and only if $\Delta u(t)$ satisfies the weak homogeneous Neumann boundary condition for all $t\in [0,T]$.
\end{rem}
\begin{proof} See \cite{SZ25}.
\end{proof}
\subsection{Stability estimates}
For the following results, the regularity $u\in L^2(\Omega;\mathcal{C}([0,T];H^2(\Lambda)))$ of the variational solution to equation \eqref{equation} is crucial. It is provided by Proposition \ref{240312_prop1}, when \ref{A1} - \ref{A5} are satisfied.
\begin{lem}\label{Lemma 220923_01}
Under assumptions \ref{A1} - \ref{A5}, there exists a constant 
\[C_4=C_4(T,\Lambda,\Vert u \Vert_{L^2(\Omega; \mathcal{C}([0,T]; H^2(\Lambda)))}, \|\mathbf{v} \|_{L^{\infty}}, g , f, \beta )>0\] such that for all $s,t \in [0,T]$:
\begin{align*}
\erww{\Vert u(t) - u(s) \Vert_{L^2(\Lambda)}^2} \leq C_4 \, |t-s|.
\end{align*}
\end{lem}
\begin{proof}
For any fixed  $s \in [0,T]$, then we have for each $t \in [0,T]$
\begin{align}\label{Eq L2.1}
v(t):= u(t) - u(s) = \int_s^t \big( \Delta u(r) - {\operatorname{div}} (\mathbf{v} f(u(r))) + \beta(u(r)) \big) \, dr + \int_s^t g(u(r)) \, dW.
\end{align}
 We apply It\^{o} formula to the functional $x\mapsto \frac{1}{2} \|x\|_{L^2}^2$ on $v$ to have
\begin{align*}
\frac{1}{2} \Vert v(t) \Vert_{L^2}^2 = \int_s^t (\Delta u(r), v(r))_2 \,dr & - \int_s^t ({\operatorname{div}} (\mathbf{v} f(u(r))), v(r))_2 \,dr + \int_s^t (\beta (u(r)), v(r))_2 \,dr   \\ 
&  + \int_s^t (v(r), g(u(r)) \, dW)_2 + \frac{1}{2} \int_s^t \Vert g(u(r)) \Vert_{L^2}^2 \, dr.
\end{align*}
Applying the expectation and  Young's inequality along with the assumptions \ref{A1}-\ref{A5} we obtain:
\begin{align*}
& \frac{1}{2} \erww{\Vert v(t) \Vert_{L^2}^2} \\
& \leq \frac{1}{2} \erww{\int_s^t \Vert \Delta u(r) \Vert_{L^2}^2 \, dr} +  \frac{1}{2} \erww{\int_s^t \Vert v(r)\Vert_{L^2}^2 \,dr}  + \frac{1}{2}  (L_f \, \|\mathbf{v} \|_{L^{\infty}}^{2} + L_{\beta}) \erww{\int_s^t \Vert u(r) \Vert_{L^2}^2 \, dr}   \\
& \quad + \frac{1}{2} \Vert \nabla u \Vert_{L^2(\Omega; \mathcal{C}([0,T]; L^2(\Lambda)))}^2 |t-s| + \frac{1}{2} \erww{\int_s^t \Vert v(r)\Vert_{L^2}^2 \,dr}  +\frac{1}{2} C_g \erww{\int_s^t |\Lambda| + \Vert u(r) \Vert_{L^2}^2 \, dr} \\
&  \leq  \frac{1}{2}\Vert \Delta u \Vert_{L^2(\Omega, \mathcal{C}([0,T]; L^2(\Lambda)))}^2 |t-s|+ \erww{\int_s^t \Vert v(r)\Vert_{L^2}^2 \,dr}  + \frac{1}{2} \Vert \nabla u \Vert_{L^2(\Omega; \mathcal{C}([0,T]; L^2(\Lambda)))}^2 |t-s| \\
& \quad+\frac{1}{2} C_g |\Lambda| |t-s| + \frac{1}{2} (L_f \, \|\mathbf{v} \|_{L^{\infty}}^{2} + L_{\beta} + C_g ) \Vert u\Vert_{L^2(\Omega, \mathcal{C}([0,T]; L^2(\Lambda)))}^2 |t-s| \\
 & = C_4 |t-s| +  \erww{\int_s^t \Vert v(r)\Vert_{L^2}^2 \,dr},
\end{align*}
where $C_4=C_4(T,\Lambda,\Vert u \Vert_{L^2(\Omega; \mathcal{C}([0,T]; H^2(\Lambda)))}, \|\mathbf{v} \|_{L^{\infty}}, g , f, \beta )>0$. Now, Gronwall's lemma yields the assertion.
\end{proof}
\begin{lem}\label{Lemma 220923_02}
Under assumptions \ref{A1} - \ref{A5}, there exists a constant
\[C_5=C_5(T,\Lambda,\Vert u \Vert_{L^2(\Omega; \mathcal{C}([0,T]; H^2(\Lambda)))}, \Vert g' \Vert_{\infty}, \|\mathbf{v} \|_{L^{\infty}},  \|f' \|_{L^{\infty}}, \beta)>0\] 
such that for all $s,t \in [0,T]$:
\begin{align*}
\erww{\Vert \nabla(u(t) - u(s)) \Vert_{L^2(\Lambda)}^2} \leq C_5 \,|t-s|.
\end{align*}
\end{lem}
\begin{proof}
For any fixed $s \in [0,T]$, we set $v(t):= u(t) - u(s)$ for each $t \in [0,T]$. Now, we apply It\^{o} formula to the functional $x\mapsto \|x\|_{L^2}^2$ on $\nabla v$ for \eqref{Eq L2.1} to obtain
\begin{align*}
\frac{1}{2} \Vert \nabla v(t) \Vert_2^2 = &- \int_s^t (\Delta u(r), \Delta v(r))_2 \,dr + \int_s^t ({\operatorname{div}} (\mathbf{v} f(u(r))), \Delta v(r))_2 \,dr \\
&  - \int_s^t (\beta (u(r)), \Delta v(r))_2 \,dr - \int_s^t (\Delta v(r), g(u(r)) \, dW)_2 \\
& - \frac{1}{2} \int_s^t (g(u(r)), \Delta g(u(r))_2 \, dr. 
\end{align*}
Taking the expectation on both sides of the equality, applying Gauss' theorem and Young's inequality along with the assumptions \ref{A2}-\ref{A5} yields
\begin{align*}
& \frac{1}{2} \erww{\Vert \nabla v(t) \Vert_2^2}  \\
 & \leq- \erww{\int_s^t (\Delta u(r), \Delta u(r) - \Delta u(s))_2 \,dr} + \frac{1}{2}  \|\mathbf{v} \|_{L^{\infty}}^2  \|f' \|_{L^{\infty}}^2 \erww{\int_s^t \Vert \nabla u(r) \Vert_2^2 \, dr} \\
& \quad +  \erww{\int_s^t \Vert \Delta (u(r) - \Delta u(s)) \Vert_2^2 \, dr} + \frac{1}{2} \, C_{\beta} \, \erww{\int_s^t \Vert u(r) \Vert_2^2 \, dr}  + \frac{1}{2} \erww{\int_s^t \Vert \nabla g(u(r)) \Vert_2^2 \, dr} \\
& \leq \, \erww{\int_s^t (\Delta u(r), \Delta u(s))_2 \,dr} -  \erww{\int_s^t \Vert \Delta u(r) \Vert_2^2 \, dr} + \frac{1}{2}  \|\mathbf{v} \|_{L^{\infty}}^2  \|f' \|_{L^{\infty}}^2 \erww{\int_s^t \Vert \nabla u(r) \Vert_2^2 \, dr}\\
& \quad  + 2 \erww{\int_s^t \Vert \Delta u(r) \Vert_2^2 \, dr} + 2 \erww{\int_s^t \Vert \Delta u(s) \Vert_2^2 \, dr}  + \frac{1}{2} \, C_{\beta} \, \erww{\int_s^t \Vert u(r) \Vert_2^2 \, dr} \\
& \quad + \frac{1}{2} \Vert g' \Vert_{\infty}^2 \erww{\int_s^t \Vert \nabla u(r) \Vert_2^2 \, dr} \\
& \leq \, \frac{5}{2} \Vert \Delta u \Vert_{L^2(\Omega; \mathcal{C}([0,T]; L^2(\Lambda)))}^2 |t-s| + \frac{1}{2} \,  C_{\beta} \, \Vert  u \Vert_{L^2(\Omega; \mathcal{C}([0,T]; L^2(\Lambda)))}^2 |t-s|\\
& \quad  +  \frac{3}{2} \erww{\int_s^t \Vert \Delta u(r) \Vert_2^2 \, dr} + \frac{1}{2} (  \|\mathbf{v} \|_{L^{\infty}}^2  \|f' \|_{L^{\infty}}^2  + \Vert g' \Vert_{\infty}^2 \Vert) \nabla u \Vert_{L^2(\Omega; \mathcal{C}([0,T]; L^2(\Lambda)))}^2 |t-s| \\
 &  \leq \, C_5 |t-s|,
\end{align*}
where $C_5=C_5(T,\Lambda,\Vert u \Vert_{L^2(\Omega; \mathcal{C}([0,T]; H^2(\Lambda)))}, \Vert g' \Vert_{\infty}, \|\mathbf{v} \|_{L^{\infty}},  \|f' \|_{L^{\infty}}, \beta)>0$.
\end{proof}
\section{Convergence rates}\label{Convergence rates}
From now on we assume that \ref{A1} to \ref{A5} are satisfied and $u$ is a solution to \eqref{equation}. Then, by Proposition \ref{240312_prop1}, $u$ is in $L^2(\Omega; \mathcal{C}([0,T]; H^2(\Lambda)))$ and satisfies the homogeneous Neumann boundary condition in the weak sense. Moreover, let us assume that $(v_N^0)_{N \in \N} \subset L^2(\Omega; H^3(\Lambda)) \cap L^{16}(\Omega; H^1(\Lambda))$ is a sequence such that
\begin{itemize}
    \item[$i.)$] $\sup\limits_{N \in \N} \erww{\Vert v_N^0 \Vert_{H^3}^2 + \Vert v_N^0 \Vert_{H^1}^{16}} < \infty$,
    \item[$ii.)$] $\sup\limits_{N \in \N} \frac{1}{N}\erww{\Vert \nabla \Delta v_N^0 \Vert_2^2} < \infty$,
    \item[$iii.)$] $\sup\limits_{N \in \N}N \, \erww{\Vert v_N^0 - u_0 \Vert_2^2} < \infty$.
\end{itemize}
The existence of such a sequence is ensured by \cite[Theorem 5.33]{AF03}. A suitable choice is $v_N^0 := \rho_{N^{-1/2}} \ast u_0$, where $\rho$ is a standard mollifier on $\R^{d}$ and $\rho_{\varepsilon}(x) := \varepsilon^{-d} \rho(x/\varepsilon)$.\\
Now, for $N \in \N$ and $n \in \{1,...,N\}$ let $v^n = v_N^n$ be the solution to \eqref{ES} with initial value $v_N^0$. We set $\tilde{N}_0 := \max(N_1, N_2, N_3)$, where $N_1, N_2$ and $N_3$ are the numbers in Lemma \ref{Lemma 250923_01},  Corollary \ref{Corollary 241125_01} and Lemma \ref{Lemma 250923_03}, respectively. Thanks to the assumptions $i.), ii.)$ and $iii.)$ the boundedness results for $v^n$ in Lemma \ref{Lemma 250923_01},  Corollary \ref{Corollary 241125_01} and Lemma \ref{Lemma 250923_03} still hold true uniformly for all $N \geq \tilde{N}_0$.
\begin{lem}\label{Lemma 220923_03}
There exists $\hat{N}_0 \in \N$ with $\hat{N}_0 \geq \tilde{N}_0$ and a constant $C_6>0$ not depending on $N$ such that for all $N \geq \hat{N}_0$:
\begin{align*}
\sup_{n\in \{1,...,N\}} \erww{\Vert u(t_n) - v^n \Vert_2^2} + \tau \sum\limits_{n=1}^N \erww{\Vert \nabla u(t_n) - \nabla v^n) \Vert_2^2} \leq C_6 \tau .
\end{align*}
\end{lem}
\begin{proof}
We set $e^n:= u(t_n) - v^n$. Then, we have
\begin{align*}
e^n - e^{n-1} & = \int_{t_{n-1}}^{t_n} \Delta (u(s)- v^n) \, ds - \int_{t_{n-1}}^{t_n}\, {\operatorname{div}} (\mathbf{v} f(u(s))) - {\operatorname{div}} (\mathbf{v}^n f(v^n)) \, ds\\
& \quad + \int_{t_{n-1}}^{t_n} \beta(u(s)) - \beta (v^{n})\, ds + \int_{t_{n-1}}^{t_n}g(u(s)) - g(v^{n-1})\, dW \\
&= \int_{t_{n-1}}^{t_n} \Delta (u(s)- u(t_n)) \, ds  + \int_{t_{n-1}}^{t_n} \Delta e^n \, ds\\
& \quad - \int_{t_{n-1}}^{t_n}\, {\operatorname{div}} (\mathbf{v} f(u(s))) - {\operatorname{div}} (\mathbf{v} f(u(t_n))) + {\operatorname{div}} (\mathbf{v} f(u(t_n))) - {\operatorname{div}} (\mathbf{v} f(v^{n}))  \, ds \\
&\quad  +  \int_{t_{n-1}}^{t_n}\, {\operatorname{div}} (\mathbf{v} f(v^{n}))  - {\operatorname{div}} (\mathbf{v}^n f(v^n)) \, ds \\
& \quad + \int_{t_{n-1}}^{t_n} \beta(u(s)) - \beta (u(t_n))\, ds + \int_{t_{n-1}}^{t_n} \beta(u(t_n)) - \beta (v^{n})\, ds \\
&\quad + \int_{t_{n-1}}^{t_n}g(u(s)) - g(u(t_{n-1}))\, dW + \int_{t_{n-1}}^{t_n}g(u(t_{n-1})) - g(v^{n-1})\, dW.
\end{align*}
Now, we multiply with $e^n$ and integrate over $\Lambda$ to obtain
\begin{align*}
\frac{1}{2} &\bigg( \Vert e^n \Vert_2^2 - \Vert e^{n-1} \Vert_2^2 + \Vert e^n - e^{n-1} \Vert_2^2\bigg) = (e^n, e^n - e^{n-1})_2 \\
=& \int_{t_{n-1}}^{t_n} \int_{\Lambda} \Delta (u(s)- u(t_n)) \, e^n \, ds  + \int_{t_{n-1}}^{t_n} \int_{\Lambda} \Delta e^n \cdot e^n\, ds\\
&- \int_{t_{n-1}}^{t_n} \int_{\Lambda} ({\operatorname{div}} (\mathbf{v} f(u(s))) - {\operatorname{div}} (\mathbf{v} f(u(t_n))))\, e^n\, ds  \\
&  + \int_{t_{n-1}}^{t_n} \int_{\Lambda} ({\operatorname{div}} (\mathbf{v} f(u(t_n))) - {\operatorname{div}} (\mathbf{v} f(v^{n})))\, e^n\, ds \\
&+ \int_{t_{n-1}}^{t_n} \int_{\Lambda} ({\operatorname{div}} (\mathbf{v} f(v^{n}))  - {\operatorname{div}} (\mathbf{v}^n f(v^n)))\, e^n\, ds \\
&+ \int_{t_{n-1}}^{t_n} \int_{\Lambda}  (\beta(u(s)) - \beta (u(t_n)) ) \,e^n\, ds  + \int_{t_{n-1}}^{t_n} \int_{\Lambda} ( \beta(u(t_n)) - \beta (v^{n})) \,e^n\, ds \\
&+ \int_{t_{n-1}}^{t_n} \int_{\Lambda}(g(u(s)) - g(u(t_{n-1})))\, e^n\, dW + \int_{t_{n-1}}^{t_n} \int_{\Lambda} (g(u(t_{n-1})) - g(v^{n-1}))\, e^n\, dW.
\end{align*}
Taking expectation we obtain the equality
\begin{align}\label{eq 220923_01}
\frac{1}{2} \erww{ \Vert e^n \Vert_2^2 - \Vert e^{n-1} \Vert_2^2 + \Vert e^n - e^{n-1} \Vert_2^2} = \sum_{i=1}^9 \mathcal{T}_{i},
\end{align}
where 
\begin{align*}
      \mathcal{T}_{1} &= \erww{\int_{t_{n-1}}^{t_n} \int_{\Lambda} \Delta (u(s)- u(t_n)) \, e^n \, ds },  \\  
    \mathcal{T}_{2} &= \erww{ \int_{t_{n-1}}^{t_n} \int_{\Lambda} \Delta e^n \cdot e^n\, ds }, \\
     \mathcal{T}_{3} &= -\erww{\int_{t_{n-1}}^{t_n} \int_{\Lambda} ({\operatorname{div}} (\mathbf{v} f(u(s))) - {\operatorname{div}} (\mathbf{v} f(u(t_n))))\, e^n\, ds},  \\
     \mathcal{T}_{4} & = \erww{\int_{t_{n-1}}^{t_n} \int_{\Lambda} ({\operatorname{div}} (\mathbf{v} f(u(t_n))) - {\operatorname{div}} (\mathbf{v} f(v^{n})))\, e^n\, ds  } ,\\
     \mathcal{T}_{5} & = \erww{\int_{t_{n-1}}^{t_n} \int_{\Lambda} ({\operatorname{div}} (\mathbf{v} f(v^{n}))  - {\operatorname{div}} (\mathbf{v}^n f(v^n)))\, e^n\, ds}, \\
     \mathcal{T}_{6} &= \erww{\int_{t_{n-1}}^{t_n} \int_{\Lambda}  (\beta(u(s)) - \beta (u(t_n)) ) \,e^n\, ds },  \\  
         \mathcal{T}_{7} &= \erww{ \int_{t_{n-1}}^{t_n} \int_{\Lambda} ( \beta(u(t_n)) - \beta (v^{n})) \,e^n\, ds } ,
\end{align*}
\begin{align*}
     \mathcal{T}_{8} &= \erww{\int_{t_{n-1}}^{t_n} \int_{\Lambda}(g(u(s)) - g(u(t_{n-1})))\, e^n\, dW}, \\
    \mathcal{T}_{9} &= \erww{\int_{t_{n-1}}^{t_n} \int_{\Lambda} (g(u(t_{n-1})) - g(v^{n-1}))\, e^n\, dW} .
\end{align*}
Now, using Young's inequality, Lemma \ref{Lemma 250923_01}, Lemma \ref{Lemma 220923_01}, Lemma \ref{Lemma 220923_02} and It\^{o} isometry, for all $N \geq \tilde{N}_0$, it yields
\begin{align*}
   |\mathcal{T}_{1}| &\leq \erww{\bigg| \int_{t_{n-1}}^{t_n} \int_{\Lambda} \nabla (u(s)- u(t_n)) \nabla e^n \, ds \bigg|} \\
& \leq \erww{ \int_{t_{n-1}}^{t_n} \Vert \nabla (u(s)- u(t_n)) \Vert_2^2 \, ds} + \frac{1}{4} \erww{\int_{t_{n-1}}^{t_n}  \Vert \nabla e^n \Vert_2^2 \, ds} \leq C_5 \tau^2 + \frac{\tau}{4} \erww{\Vert \nabla e^n \Vert_2^2}, \\
\mathcal{T}_{2} &= -\tau \erww{\Vert \nabla e^n \Vert_2^2}, \\
|\mathcal{T}_{3}|  &\leq  \erww{\left|\int_{t_{n-1}}^{t_n} \int_{\Lambda} ({\operatorname{div}} (\mathbf{v} f(u(s))) - {\operatorname{div}} (\mathbf{v} f(u(t_n))))\, e^n\, dx \, ds \right|} \\
&\leq 2\Vert \mathbf{v} \Vert_{\infty}^2 L_f^2 \erww{\int_{t_{n-1}}^{t_n} \Vert u(s) - u(t_{n}) \Vert_2^2 \, ds} + \frac{1}{8} \erww{\int_{t_{n-1}}^{t_n} \Vert \nabla e^n \Vert_2^2 \, ds} \\
&\leq 2\Vert \mathbf{v} \Vert_{\infty}^2 L_f^2 \, C_4 \tau^2 + \frac{\tau}{8} \erww{\Vert \nabla e^n \Vert_2^2},  \\
|\mathcal{T}_{4}| &\leq \erww{\left| \int_{t_{n-1}}^{t_n} \int_{\Lambda} ({\operatorname{div}} (\mathbf{v} f(u(t_n)))) - {\operatorname{div}} (\mathbf{v} f(v^{n})))\, e^n\, dx \, ds \right|}\\
&\leq 2\Vert \mathbf{v} \Vert_{\infty}^2 L_f^2 \erww{\int_{t_{n-1}}^{t_n} \Vert u(t_{n}) - v^{n} \Vert_2^2 \, ds} + \frac{1}{8} \erww{\int_{t_{n-1}}^{t_n} \Vert \nabla e^n \Vert_2^2 \, ds}
\\
&= 2\Vert \mathbf{v} \Vert_{\infty}^2 L_f^2 \,\tau \erww{\Vert e^{n} \Vert_2^2} + \frac{\tau}{8} \erww{\Vert \nabla e^n \Vert_2^2},  \\
|\mathcal{T}_{5}| &\leq \erww{\left|\int_{t_{n-1}}^{t_n} \int_{\Lambda} ({\operatorname{div}} (\mathbf{v} f(v^{n}))  - {\operatorname{div}} (\mathbf{v}^n f(v^n)))\, e^n\, dx \, ds \right|} \\
&\leq 2L_f \, \tau \Vert \mathbf{v} - \mathbf{v}^n \Vert_{\infty}^2  C_1
 + \frac{\tau}{4} \erww{\Vert \nabla e^n \Vert_2^2} \leq 2L_f \, \tau^3 \Vert D\mathbf{v}\Vert_{\infty}^2  C_1
 + \frac{\tau}{4} \erww{\Vert \nabla e^n \Vert_2^2} ,\\ 
 |\mathcal{T}_{6}| &= \erww{\int_{t_{n-1}}^{t_n} \int_{\Lambda}  (\beta(u(s)) - \beta (u(t_n)) ) \,e^n\, ds} \\
& \leq \Vert \beta^{'} \Vert_{\infty}^2 \erww{\int_{t_{n-1}}^{t_n} \Vert u(s) - u(t_{n}) \Vert_2^2 \, dx \, ds} + \frac{\tau}{4} \erww{\Vert e^n \Vert_2^2} \leq \Vert \beta^{'} \Vert_{\infty}^2 C_4 \tau^2 + \frac{\tau}{4} \erww{\Vert e^n \Vert_2^2}, \\
\mathcal{T}_{7} &=  \erww{\int_{t_{n-1}}^{t_n} \int_{\Lambda} ( \beta(u(t_n)) - \beta (v^{n})) \,e^n\, dx \, ds}  \leq \Vert \beta^{'} \Vert_{\infty} \erww{\int_{t_{n-1}}^{t_n} \Vert e^{n} \Vert_2^2 \, ds}\\
&= \Vert \beta^{'} \Vert_{\infty} \tau \erww{\Vert e^{n} \Vert_2^2},
\end{align*}
\begin{align*}
|\mathcal{T}_{8}| &\leq 2\erww{\int_{t_{n-1}}^{t_n} \Vert g(u(s)) - g(u(t_{n-1})) \Vert_2^2 \, ds} + \frac{1}{8} \erww{\Vert e^n - e^{n-1} \Vert_2^2} \\
&\leq 2\Vert g' \Vert_{\infty}^2 \erww{\int_{t_{n-1}}^{t_n} \Vert u(s) - u(t_{n-1}) \Vert_2^2 \, ds} + \frac{1}{8} \erww{\Vert e^n - e^{n-1} \Vert_2^2} \\
&\leq 2\Vert g' \Vert_{\infty}^2 C_4 \tau^2 + \frac{1}{8} \erww{\Vert e^n - e^{n-1} \Vert_2^2}, \\
|\mathcal{T}_{9}| &\leq 2\erww{\int_{t_{n-1}}^{t_n} \Vert g(u(t_{n-1})) - g(v^{n-1}) \Vert_2^2 \, ds} + \frac{1}{8} \erww{\Vert e^n - e^{n-1} \Vert_2^2} \\
&\leq 2\Vert g' \Vert_{\infty}^2 \tau \erww{\Vert e^{n-1} \Vert_2^2} + \frac{1}{8} \erww{\Vert e^n - e^{n-1} \Vert_2^2}.
\end{align*}
Exchanging the indices $n$ and $m$ and summing over $m=1,...,n$, $n \in \{1,...,N\}$, equality \eqref{eq 220923_01} and the previous estimates yield the existence of a constant $\tilde{C}>0$ only depending on $T, v, f, \beta, g, C_1, C_4$ and $C_5$ such that
\begin{align*}
&\frac{1}{2} \erww{\Vert e^n \Vert_2^2} + \frac{1}{4} \sum\limits_{m=1}^n \erww{\Vert e^m - e^{m-1} \Vert_2^2} + \frac{\tau}{4} \sum\limits_{m=1}^n \erww{\Vert \nabla e^m \Vert_2^2}\\
&\leq \frac{1}{2} \erww{\Vert e^0 \Vert_2^2} + \tilde{C} \tau + 2\Vert g' \Vert_{\infty}^2 \tau \sum\limits_{m=1}^n \erww{\Vert e^{m-1} \Vert_2^2} + \left(2\Vert \mathbf{v} \Vert_{\infty}^2 L_f^2 + \Vert \beta' \Vert_{\infty}^2 + \frac{1}{4} \right) \,\tau \sum\limits_{m=1}^n \erww{\Vert e^{m} \Vert_2^2} \\
&\leq \frac{1}{2} \erww{\Vert e^0 \Vert_2^2} + \tilde{C} \tau + \left( 2\Vert g' \Vert_{\infty}^2 + 2\Vert \mathbf{v} \Vert_{\infty}^2 L_f^2 + \Vert \beta' \Vert_{\infty}^2 + \frac{1}{4} \right)\tau \sum\limits_{m=1}^n \erww{\Vert e^{m-1} \Vert_2^2} \\
& \quad + \left(2\Vert \mathbf{v} \Vert_{\infty}^2 L_f^2 + \Vert \beta' \Vert_{\infty}^2 + \frac{1}{4} \right) \,\tau \erww{\Vert e^{n} \Vert_2^2}. 
\end{align*}
We can find $\hat{N}_0 \in \N$ with $\hat{N}_0 \geq \tilde{N}_0$ such that $(2\Vert \mathbf{v} \Vert_{\infty}^2 L_f^2 + \Vert \beta' \Vert_{\infty}^2 + 1/4)\tau \leq 1/4$ for all $N \geq \hat{N}_0$. Then, by using the estimate
\begin{equation*}
    \erww{\Vert e^0 \Vert_2^2}= \erww{\Vert v_N^0 - u_0 \Vert_2^2} \leq C \tau
\end{equation*}
for some constant $C>0$. For all $N \geq \hat{N}_0$ we obtain 
\begin{align*}
    &\frac{1}{4} \erww{\Vert e^n \Vert_2^2} + \frac{1}{4} \sum\limits_{m=1}^n \erww{\Vert e^m - e^{m-1} \Vert_2^2} + \frac{\tau}{4} \sum\limits_{m=1}^n \erww{\Vert \nabla e^m \Vert_2^2}\\
    &\leq (\frac{C}{2} + \tilde{C})\tau + \left( 2\Vert g' \Vert_{\infty}^2 + 2\Vert \mathbf{v} \Vert_{\infty}^2 L_f^2 + \Vert \beta' \Vert_{\infty}^2 + \frac{1}{4} \right)\tau \sum\limits_{m=1}^n \erww{\Vert e^{m-1} \Vert_2^2} .
\end{align*}
Now, Lemma \ref{Gronwall} yields the assertion.
\end{proof}
\begin{prop}\label{281123_01}
Let $w \in H^2(\Lambda)$ and $\mathcal{T}$ an admissible mesh. Then, there exists a unique solution $(\tilde{w}_K)_{K \in \mathcal{T}}$ to the following discrete problem:\\
Find $(\tilde{w}_K)_{K \in \mathcal{T}}$ such that 
\begin{align*}
\sum\limits_{K \in \mathcal{T}} m_K \tilde{w}_K = \int_{\Lambda} w \, dx
\end{align*}
and
\begin{align*}
\sum\limits_{\sigma \in \mathcal{E}_K^{int}} \frac{m_{\sigma}}{d_{K|L}} (\tilde{w}_K - \tilde{w}_L) = - \int_{K} \Delta w \, dx, ~~~\forall K \in \mathcal{T}.
\end{align*}
\end{prop}
\begin{proof}
See e.g. \cite{Flore21}, Definition 6 and Remark 2.
\end{proof}
\begin{defn}\label{Definition parabolic projection}
Let $\mathcal{T}$ be an admissible mesh and $w \in H^2(\Lambda)$. The elliptic projection of $w$ on $\Tau$ is given by
\begin{align*}
\tilde{w}:= \sum\limits_{K \in \mathcal{T}} \tilde{w}_K \mathds{1}_K,
\end{align*}
where $(\tilde{w}_K)_{K \in \mathcal{T}}$ are given by Proposition \ref{281123_01}. The centered projection of $w$ on $\Tau$ is given by
\begin{align*}
\hat{w}:= \sum\limits_{K \in \mathcal{T}} w(x_K) \mathds{1}_K.
\end{align*}
\end{defn}
\begin{lem}\label{Lemma 260923_01}
Let $\Tau_h$ be an admissible mesh with $\operatorname{size}(\mathcal{T}_h)=h$ and $w$ a random variable with values in $H^2(\Lambda)$.\\
Moreover, let $\tilde{w}_h$ be the elliptic projection and $\hat{w}_h$ the centered projection of $w$ a.s. in $\Omega$. Then there exists a constant $C_7>0$ only depending on $\Lambda$ and $\operatorname{reg}(\Tau_h)$ such that $\mathbb{P}$-a.s. in $\Omega$:
\begin{align*}
\Vert w - \tilde{w}_h \Vert_2^2 &\leq C_7 h^2 \Vert w \Vert_{H^2(\Lambda)}^2, \\
\Vert w - \hat{w}_h \Vert_2^2 &\leq C_7 h^2 \Vert w \Vert_{H^2(\Lambda)}^2, \\
|\hat{w}_h - \tilde{w}_h |_{1,h}^2 &\leq C_7 h^2 \Vert w \Vert_{H^2(\Lambda)}^2.
\end{align*}
\end{lem}
Remark that by \eqref{291123_01} there exists a constant $C_8>0$ only depending on $\Lambda$ and $\operatorname{reg}(\Tau_h)$ such that $\mathbb{P}$-a.s. in $\Omega$
\begin{align*}
\Vert \hat{w}_h - \tilde{w}_h \Vert_2^2 \leq C_8 h^2 \Vert w \Vert_{H^2(\Lambda)}^2.
\end{align*}
Moreover, if \eqref{mesh_regularity} is satisfied, $C_7$ and $C_8$ may depend on $\chi$ and the dependence of $C_7$ and $C_8$ on $\operatorname{reg}(\Tau_h)$ can be omitted.
\begin{proof}
See \cite{SZ25}
\end{proof}
\begin{cor}\label{Corollary 260923_01}
Let $\tilde{v}_h^n$ be the elliptic projection of $v^n$, a.s. in $\Omega$. Then there exists a constant $C_9>0$ such that
\begin{align*}
\sup_{n\in \{1,...,N\}} \erww{\Vert v^n - \tilde{v}_h^n \Vert_2^2} \leq C_9 h^2.
\end{align*}
Especially, we have
\begin{align*}
\tau \sum\limits_{n=1}^{N} \erww{\Vert v^n - \tilde{v}_h^n \Vert_2^2} \leq C_9 T h^2.
\end{align*}
\end{cor}
\begin{proof}
See \cite{SZ25}
%From Lemma \ref{Lemma 260923_01} and \eqref{eq 041023_02} it follows that for all $n=1,...,N$ we have
%\begin{align*}
%\mathbb{E} \Vert v^n - \tilde{v}_h^n \Vert_2^2 \leq C_7 h^2 \mathbb{E} \Vert v^n \Vert_{H^2(\Lambda)}^2 \leq C_7 C h^2 \mathbb{E} (\Vert \Delta v^n \Vert_2^2 + \Vert v^n \Vert_2^2).
%\end{align*}
%Thus, from Lemma \ref{Lemma 250923_01} and Lemma \ref{Lemma 250923_03} we get
%\begin{align*}
%\mathbb{E} \Vert v^n - \tilde{v}_h^n \Vert_2^2 \leq C_7 C (C_3 + C_1) h^2.
%\end{align*}
%Multiplying this inequality with $\tau$ and summing over $n=1,...,N$ yields the result.
\end{proof}

\begin{lem}\label{Lemma 270923_01}
    Let the assumptions \ref{A1} to \ref{A5} be satisfied with $f \equiv id$ and let $\hat{v}^n$ be the centered projection of $v^n$. Then there exists a constant $C_{10}>0$ not depending on $N,n,h$ and $N_0 \in \N$ with $N_0 \geq \hat{N}_0$ such that for all $N \geq N_0$
	\begin{align*}
		\sup_{n\in \{1,...,N\}} \mathbb{E} \Vert \hat{v}^n - u_h^n \Vert_2^2 \leq C_{10}(\tau + h^2 + \frac{h^2}{\tau}).
	\end{align*}
\end{lem}
\begin{proof}
	We divide \eqref{FVS} by $m_K$ and subtract this equality from \eqref{ES} and get
	\begin{align*}
		&\bigg( v^n - u_K^n - (v^{n-1} - u_K^{n-1}) \bigg) - \tau \left( \Delta v^n + \frac{1}{m_K} \sum\limits_{\sigma \in \mathcal{E}_K^{int}} \frac{m_{\sigma}}{d_{K|L}} (u_K^n - u_L^n) \right) \\
		&+ \tau \left( \operatorname{div}( \mathbf{v}^n v^n) - \frac{1}{m_K} \sum\limits_{\sigma \in \mathcal{E}_K^{int}} m_{\sigma} \mathbf{v}_{K, \sigma}^n u_{\sigma}^n \right)\\
        &= \left( g(v^{n-1}) - g(u_K^{n-1}) \right) \Delta_n W + \tau \left(\beta(v^n) - \beta(u_K^n) \right)~~~\text{on}~K.
	\end{align*}
	Multiplying with $v_K^n - u_K^n$, where $(v_K^n)_{K \in \Tau}$ is the centered projection of $v^n$, i.e., $v_K^n = v(t_n, x_K)$, and integrating over $K$ yields
	\begin{align*}
		I_K + II_K + III_K + IV_K = V_K + VI_K,
	\end{align*}
	where
	\begin{align*}
		I_K &= \int_K \left(v^n - v_K^n - (v^{n-1} - v_K^{n-1})\, \right) (v_K^n - u_K^n) \, dx, \\
		II_K &= \frac{1}{2} \left(\Vert v_K^n - u_K^n \Vert_{L^2(K)}^2 - \Vert v_K^{n-1} - u_K^{n-1} \Vert_{L^2(K)}^2 + \Vert v_K^n - u_K^n - (v_K^{n-1} - u_K^{n-1}) \Vert_{L^2(K)}^2 \right), \\
		III_K &= -\tau\bigg( \int_K \Delta v^n \, dx + \sum\limits_{\sigma \in \mathcal{E}_K^{int}} \frac{m_{\sigma}}{d_{K|L}} (u_K^n - u_L^n)\bigg) (v_K^n - u_K^n), \\
		%&= + \tau\bigg( \sum\limits_{\sigma \in \mathcal{E}_K^{int}} \frac{m_{\sigma}}{d_{K|L}} (\tilde{v}_K^n - u_K^n - (\tilde{v}_L^n - u_L^n))\bigg)(\tilde{v}_K^n - u_K^n) \\
        IV_K &= \int_{t_{n-1}}^{t_n} \sum\limits_{\sigma \in \mathcal{E}_K^{int}} \int_{\sigma} \left(\mathbf{v}^n v^n - \mathbf{v}(t,x) u_{\sigma}^n \right) \cdot \vec{n}_{K, \sigma} \, d\gamma(x) \, (v_K^n - u_K^n),\\
		V_K &= \int_K (g(v^{n-1}) - g(u_K^{n-1})) (v_K^n - u_K^n)  \Delta_n W \, dx, \\
        VI_K &= \tau \int_K (\beta(v^n) - \beta(u_K^n))(v_K^n - u_K^n) \, dx .
	\end{align*}
	Summing over $K \in \mathcal{T}$, using the abbreviations
    \begin{align*}
        \tilde{e}_K^n &= v_K^n - u_K^n, ~\tilde{e}_h^n = \sum\limits_{K \in \Tau} \tilde{e}_K^n \mathds{1}_K, \\
        \hat{e}_K^n &= v_K^n - \tilde{v}_K^n, ~\hat{e}_h^n = \sum\limits_{K \in \Tau} \hat{e}_K^n \mathds{1}_K, \\
        \dot{e}_K^n &= \tilde{v}_K^n - u_K^n, ~\dot{e}_h^n = \tilde{v}_h^n - u_h^n, \\
        \overline{e}_h^n &= v^n - \hat{v}_h^n,
    \end{align*}
    where $\hat{v}_h^n = \sum\limits_{K \in \Tau} v_K^n \mathds{1}_K$ and the definition of the elliptic projection it yields
	\begin{align}\label{eq 260923_01}
		I + II + III + IV = V + VI,
	\end{align}
	where %rewriting $III$ by using the discrete integration by parts rule (see Remark \ref{DIBP}), we have
	\begin{align*}
		I &= \int_{\Lambda} (\overline{e}_h^n - \overline{e}_h^{n-1}) \tilde{e}_h^n= (\overline{e}_h^n - \overline{e}_h^{n-1}, \tilde{e}_h^n)_2, \\
		II &= \frac{1}{2} \left(\Vert \tilde{e}_h^n \Vert_2^2 - \Vert \tilde{e}_h^{n-1} \Vert_2^2 + \Vert \tilde{e}_h^n - \tilde{e}_h^{n-1}\Vert_2^2 \right), \\
		III &= \tau \sum\limits_{K \in \Tau}\sum\limits_{\sigma \in \mathcal{E}_K^{int}} \frac{m_{\sigma}}{d_{K|L}} \left((\tilde{v}_K^n - \tilde{v}_L^n) - (u_K^n - u_L^n)\right)(v_K^n - u_K^n)\\
        & = \tau \sum\limits_{K \in \Tau}\sum\limits_{\sigma \in \mathcal{E}_K^{int}} \frac{m_{\sigma}}{d_{K|L}} \left(\dot{e}_K^n - \dot{e}_L^n \right)\tilde{e}_K^n \\
        & = \tau \sum\limits_{K \in \Tau}\sum\limits_{\sigma \in \mathcal{E}_K^{int}} \frac{m_{\sigma}}{d_{K|L}} \left(\dot{e}_K^n - \dot{e}_L^n \right)\dot{e}_K^n + \tau \sum\limits_{K \in \Tau}\sum\limits_{\sigma \in \mathcal{E}_K^{int}} \frac{m_{\sigma}}{d_{K|L}} \left(\dot{e}_K^n - \dot{e}_L^n \right)\hat{e}_K^n \\
        &= III_1 + III_2, \\
        IV &= \int_{t_{n-1}}^{t_n} \sum\limits_{K \in \Tau} \sum\limits_{\sigma \in \mathcal{E}_K^{int}} \int_{\sigma} \left(\mathbf{v}^n  - \mathbf{v}(t,x) \right)v^n \cdot \vec{n}_{K, \sigma} \, d\gamma(x) \, \tilde{e}_K^n\\
        &= \int_{t_{n-1}}^{t_n} \sum\limits_{K \in \Tau} \sum\limits_{\sigma \in \mathcal{E}_K^{int}} \int_{\sigma} \mathbf{v}(t,x) (v^n - u_{\sigma}^n)  \cdot \vec{n}_{K, \sigma} \, d\gamma(x) \, \tilde{e}_K^n\\
        &= IV_1 + IV_2, \\
		V &= \int_{\Lambda} (g(v^{n-1}) - g(\hat{v}_h^{n-1})) \tilde{e}_h^n  \Delta_n W \, dx + \int_{\Lambda} (g(\hat{v}_h^{n-1}) - g(u_h^{n-1}))\tilde{e}_h^n  \Delta_n W \, dx, \\
        VI &= \tau \int_{\Lambda} (\beta(v^n) - \beta(\hat{v}^n)) \tilde{e}_h^n + (\beta(\hat{v}^n) - \beta(u_h^n)) \tilde{e}_h^n \, dx.
	\end{align*}
    Thanks to Remark \ref{DIBP}, we have
    \begin{align}\label{eq 011225_01}
        \begin{aligned}
            III_1 = \tau \sum\limits_{\sigma \in \mathcal{E}^{int}} \frac{m_{\sigma}}{d_{K|L}} |\dot{e}_K^n - \dot{e}_L^n|^2 = \tau |\dot{e}_h^n|_{1,h}^2.
        \end{aligned}
    \end{align}
    Moreover, thanks to Lemma \ref{Lemma 260923_01} and Young's inequality, we have
    \begin{align}\label{eq 011225_02}
        \begin{aligned}
            |III_2| &= \tau \left|\sum\limits_{\sigma \in \mathcal{E}^{int}} \frac{m_{\sigma}}{d_{K|L}} \left(\dot{e}_K^n - \dot{e}_L^n\right)\left(\hat{e}_K^n - \hat{e}_L^n\right)\right| \\
            &\leq \tau \left(\sum\limits_{\sigma \in \mathcal{E}^{int}} \frac{m_{\sigma}}{d_{K|L}} \left(\dot{e}_K^n - \dot{e}_L^n\right)^2 \right)^{\frac{1}{2}} \left(\sum\limits_{\sigma \in \mathcal{E}^{int}} \frac{m_{\sigma}}{d_{K|L}} \left(\hat{e}_K^n - \hat{e}_L^n\right)^2 \right)^{\frac{1}{2}} \\
            &= \tau |\dot{e}_h^n|_{1,h}|\hat{e}_h^n|_{1,h} \leq \frac{\tau}{2} |\dot{e}_h^n|_{1,h}^2 + \frac{\tau}{2} |\hat{e}_h^n|_{1,h} ^2 \leq \frac{\tau}{2} |\dot{e}_h^n|_{1,h}^2 + C_7\frac{\tau}{2}h^2 \Vert v^n \Vert_{H^2}^2 \\
            &\leq \frac{\tau}{2} |\dot{e}_h^n|_{1,h}^2 + 6C_7\tau h^2 \left( \Vert \Delta v^n \Vert_{L^2}^2 + \Vert v^n \Vert_{L^2}^2\right).
        \end{aligned}
    \end{align}
    %and thanks to Lemma \ref{Lemma 250923_01}, \ref{Lemma 250923_03} and \ref{Theorem 110124_01} we may conclude
    %\begin{align}\label{eq 011225_02}
        %\begin{aligned}
            %\erww{|III_2|} \leq \frac{\tau}{2} \erww{|\dot{e}_h^n|_{1,h}^2} + 6C_7(C_1 + C_3)\tau h^2.
        %\end{aligned}
    %\end{align}    
    Using the Lemma \ref{Lemma 260923_01} for $VI$, it yields
	\begin{align}\label{eq 011225_04}
        \begin{aligned}
            |VI| &\leq \frac{\tau}{2}\Vert \beta ' \Vert_{\infty}^2 \left(\Vert \overline{e}_h^n\Vert_2^2 + \Vert \tilde{e}_h^n\Vert_2^2 + 2\Vert \tilde{e}_h^n\Vert_2^2 \right) \\
            &\leq \frac{3\tau}{2} \Vert \beta ' \Vert_{\infty}^2\Vert \tilde{e}_h^n\Vert_2^2 + \frac{\tau}{2}\Vert \beta ' \Vert_{\infty}^2 C_7h^2 \Vert v^n \Vert_{H^2}^2 \\
            &\leq \frac{3\tau}{2} \Vert \beta ' \Vert_{\infty}^2\Vert \tilde{e}_h^n\Vert_2^2 + 6C_7\tau h^2\Vert \beta ' \Vert_{\infty}^2 \left(\Vert \Delta v^n \Vert_{L^2}^2 + \Vert v^n \Vert_{L^2}^2\right).
            %&\leq \frac{3\tau}{2} \Vert \beta ' \Vert_{\infty}^2\erww{\Vert \tilde{e}_h^n\Vert_2^2} + 6C_7(C_1 + C_3)\Vert \beta ' \Vert_{\infty}^2 \tau h^2.
        \end{aligned}
	\end{align}
    For the term $IV_1$ we have
    \begin{align}\label{eq 031225_01}
        \begin{aligned}
            |IV_1| &= \left |\int_{t_{n-1}}^{t_n} \sum\limits_{K \in \Tau} \sum\limits_{\sigma \in \mathcal{E}_K^{int}} \int_{\sigma} \left(\mathbf{v}^n  - \mathbf{v}(t,x) \right)v^n \cdot \vec{n}_{K, \sigma} \, d\gamma(x) \, \tilde{e}_K^n \right|\\
            &= \left|  \int_{t_{n-1}}^{t_n} \sum\limits_{K \in \Tau} \int_K  (\mathbf{v}^n  - \mathbf{v}(t,x) ) \nabla v^n \tilde{e}_K^n \, dx \, dt\right| \\
            &\leq \Vert D \mathbf{v} \Vert_{\infty} \tau^2 \sum\limits_{K \in \Tau} \int_K | \nabla v^n| |\tilde{e}_K^n| \, dx \\
            &\leq \Vert D \mathbf{v} \Vert_{\infty} \tau^2 \int_{\Lambda} |\nabla v^n ||\tilde{e}_h^n| \, dx \\
            &\leq \frac{\tau^2}{2} \Vert D \mathbf{v} \Vert_{\infty}  ( \Vert \nabla v^n \Vert_{L^2(\Lambda)}^2 + \Vert \tilde{e}_h^n \Vert_{L^2(\Lambda)}^2).
        \end{aligned}
    \end{align}
    %and hence
    %\begin{align}
        %\begin{aligned}
            %|IV_1| \leq \frac{\tau^2}{2} \Vert D \mathbf{v} \Vert_{\infty} \left(\Vert \nabla v^n \Vert_{L^2(\Lambda)}^2 + \Vert \tilde{e}_h^n \Vert_{L^2(\Lambda)}^2)\right)
        %\end{aligned}
    %\end{align}
    Furthermore, we have $IV_2 = IV_{2,1} + IV_{2,2}$, where
    \begin{align*}
        IV_{2,1} &= \int_{t_{n-1}}^{t_n} \sum\limits_{K \in \Tau} \sum\limits_{\sigma \in \mathcal{E}_K^{int}} \int_{\sigma} \mathbf{v}(t,x) (v^n - v_{\sigma}^n)  \cdot \vec{n}_{K, \sigma} \, d\gamma(x) \, \tilde{e}_K^n, \\
        IV_{2,2} &= \int_{t_{n-1}}^{t_n} \sum\limits_{K \in \Tau} \sum\limits_{\sigma \in \mathcal{E}_K^{int}} \int_{\sigma} \mathbf{v}(t,x) (v_{\sigma}^n - u_{\sigma}^n)  \cdot \vec{n}_{K, \sigma} \, d\gamma(x) \, \tilde{e}_K^n
    \end{align*}
    and 
    \begin{align*}
        v_{\sigma}^n = \begin{cases}
            &v_K^n; ~\mathbf{v}_{k, \sigma}^n \geq 0, \\
            &v_L^n; ~\mathbf{v}_{k, \sigma}^n < 0.
        \end{cases}
    \end{align*}
    Now, we obtain
    \begin{align}\label{eq 031225_02}
        \begin{aligned}
            IV_{2,2} &= \tau \sum\limits_{K \in \Tau} \sum\limits_{\sigma \in \mathcal{E}_K^{int}} m_{\sigma} \mathbf{v}_{K, \sigma}^n (v_{\sigma}^n - u_{\sigma}^n) \tilde{e}_K^n \\
            &= \tau \sum\limits_{K \in \Tau} \sum\limits_{\sigma \in \mathcal{E}_K^{int}} m_{\sigma} (\mathbf{v}_{K, \sigma}^n)^- [v_K^n -u_K^n - (v_L^n - u_L^n)] \tilde{e}_K^n \\
            &= \tau \sum\limits_{K \in \Tau} \sum\limits_{\sigma \in \mathcal{E}_K^{int}} m_{\sigma} (\mathbf{v}_{K, \sigma}^n)^- \frac{1}{2} [|\tilde{e}_K^n|^2 - |\tilde{e}_L^n|^2 + |\tilde{e}_K^n - \tilde{e}_L^n|^2] \\
            &\geq \frac{\tau}{2} \sum\limits_{K \in \Tau} \sum\limits_{\sigma \in \mathcal{E}_K^{int}} m_{\sigma} (\mathbf{v}_{K, \sigma}^n)^- [|\tilde{e}_K^n|^2 - |\tilde{e}_L^n|^2] \\
            &= \frac{\tau}{2} \sum\limits_{K \in \Tau} \sum\limits_{\sigma \in \mathcal{E}_K^{int}} m_{\sigma} \mathbf{v}_{K, \sigma}^n  |v_{\sigma}^n - u_{\sigma}^n|^2 =0
        \end{aligned}
    \end{align}
    and
    \begin{align*}
        IV_{2,1} &= \int_{t_{n-1}}^{t_n} \sum\limits_{K \in \Tau} \sum\limits_{\sigma \in \mathcal{E}_K^{int}} \int_{\sigma} \mathbf{v}(t,x) (v^n - v_{\sigma}^n)  \cdot \vec{n}_{K, \sigma} \, d\gamma(x) \, \tilde{e}_K^n, \\
        &= \int_{t_{n-1}}^{t_n} \sum\limits_{\sigma \in \mathcal{E}^{int}} \int_{\sigma} \mathbf{v}(t,x) (v^n - v_{\sigma}^n)  \cdot \vec{n}_{K, \sigma} \, d\gamma(x) \, [\tilde{e}_K^n - \tilde{e}_L^n].
    \end{align*}
    Therefore, we may conclude
    \begin{align*}
        |IV_{2,1}| &\leq \tau \Vert \mathbf{v} \Vert_{\infty} \sum\limits_{\sigma \in \mathcal{E}^{int}} \int_{\sigma} |v^n - v_{\sigma}^n| |\tilde{e}_K^n - \tilde{e}_L^n|  \, d\gamma(x) \\
        &\leq \tau \Vert \mathbf{v} \Vert_{\infty} \sum\limits_{\sigma \in \mathcal{E}^{int}} \left(\int_{\sigma} |v^n - v_{\sigma}^n|^2 \, d\gamma(x) \right)^{\frac{1}{2}} \left( \int_{\sigma} |\tilde{e}_K^n - \tilde{e}_L^n|^2  \, d\gamma(x) \right)^{\frac{1}{2}} \\
         &\leq 2\tau \Vert \mathbf{v} \Vert_{\infty}^2 \sum\limits_{\sigma \in \mathcal{E}^{int}} d_{K|L}\int_{\sigma} |v^n - v_{\sigma}^n|^2 \, d\gamma(x) + \frac{\tau}{8} \sum\limits_{\sigma \in \mathcal{E}^{int}}\frac{m_{\sigma}}{d_{K|L}}|\tilde{e}_K^n - \tilde{e}_L^n|^2 \\
        &\leq 2\tau h^2 \Vert \mathbf{v} \Vert_{\infty}^2 \Vert \nabla v^n \Vert_{\infty}^2\sum\limits_{\sigma \in \mathcal{E}^{int}} m_{\sigma} d_{K|L} + \frac{\tau}{8} \sum\limits_{\sigma \in \mathcal{E}^{int}} \frac{m_{\sigma}}{d_{K|L}}|\tilde{e}_K^n - \tilde{e}_L^n|^2 \\
        &= 2\tau h^2 \Vert \mathbf{v} \Vert_{\infty}^2 \Vert \nabla v^n \Vert_{\infty}^2 |\Lambda| + \frac{\tau}{8} |\tilde{e}_h^n|_{1,h}^2.
    \end{align*}
    Since $H^2(\Lambda) \hookrightarrow \mathcal{C}(\overline{\Lambda})$ and $\nabla v^n$ satisfies the homogeneous Neumann boundary condition, we get
    \begin{align*}
        \Vert \nabla v^n \Vert_{\infty}^2 &\leq C \Vert \nabla v^n \Vert_{H^2}^2 \leq 12 C \left( \Vert \Delta \nabla v^n \Vert_{L^2(\Lambda)}^2 + \Vert \nabla v^n \Vert_{L  ^2(\Lambda)}^2 \right) \\
        &= 12 C \left( \Vert \nabla \Delta v^n \Vert_{L^2(\Lambda)}^2 + \Vert \nabla v^n \Vert_{L  ^2(\Lambda)}^2 \right) 
    \end{align*}
    for a constant $C>0$, we obtain 
    \begin{align}\label{eq 031225_03}
        \begin{aligned}
            |IV_{2,1}| &\leq 2\tau h^2 \Vert \mathbf{v} \Vert_{\infty}^2 \Vert \nabla v^n \Vert_{\infty}^2 |\Lambda| + \frac{\tau}{8} |\tilde{e}_h^n|_{1,h}^2\\
            &\leq 24C\tau h^2 \Vert \mathbf{v} \Vert_{\infty}^2  |\Lambda| \left(\Vert \Delta \nabla v^n \Vert_{L^2}^2 + \Vert \nabla v^n \Vert_{L^2}^2\right) + \frac{\tau}{8} |\tilde{e}_h^n|_{1,h}^2.
        \end{aligned}
    \end{align}
    Now, using estimates \eqref{eq 011225_01} to \eqref{eq 031225_03} in \eqref{eq 260923_01} we obtain $\mathbb{P}$-a.s.
    \begin{align*}
        &(\overline{e}_h^n - \overline{e}_h^{n-1}, \tilde{e}_h^n)_2 + \frac{1}{2} \Vert \tilde{e}_h^n \Vert_2^2 - \frac{1}{2}\Vert \tilde{e}_h^{n-1} \Vert_2^2 + \frac{1}{2}  \Vert \tilde{e}_h^n - \tilde{e}_h^{n-1}\Vert_2^2 + \tau |\dot{e}_h^n|_{1,h}^2 \\
        &\leq \frac{\tau}{2} |\dot{e}_h^n|_{1,h}^2 + 6C_7\tau h^2 \left( \Vert \Delta v^n \Vert_{L^2}^2 + \Vert v^n \Vert_{L^2}^2\right) + \frac{3\tau}{2} \Vert \beta ' \Vert_{\infty}^2\Vert \tilde{e}_h^n\Vert_2^2 \\
        & \quad + 6C_7\tau h^2\Vert \beta ' \Vert_{\infty}^2 \left(\Vert \Delta v^n \Vert_{L^2}^2 + \Vert v^n \Vert_{L^2}^2\right) + \frac{\tau^2}{2} \Vert D \mathbf{v} \Vert_{\infty} \left(\Vert \nabla v^n \Vert_{L^2}^2 + \Vert \tilde{e}_h^n \Vert_{L^2}^2\right)\\
        & \quad+ 24C\tau h^2 \Vert \mathbf{v} \Vert_{\infty}^2  |\Lambda| \left(\Vert \Delta \nabla v^n \Vert_{L^2}^2 + \Vert \nabla v^n \Vert_{L^2}^2\right) + \frac{\tau}{8} |\tilde{e}_h^n|_{1,h}^2\\
        & \quad+ \int_{\Lambda} (g(v^{n-1}) - g(\hat{v}_h^{n-1})) \tilde{e}_h^n  \Delta_n W \, dx + \int_{\Lambda} (g(\hat{v}_h^{n-1}) - g(u_h^{n-1}))\tilde{e}_h^n  \Delta_n W \, dx.
    \end{align*}
    We use that $\tilde{e}_h^n = \dot{e}_h^n + \hat{e}_h^n$ and therefore 
    \begin{align*}
        |\tilde{e}_h^n|_{1,h}^2 &\leq 2|\dot{e}_h^n|_{1,h}^2 + 2|\hat{e}_h^n|_{1,h}^2 \leq 2|\dot{e}_h^n|_{1,h}^2 + 2C_7h^2 \Vert v^n \Vert_{H^2}^2 \\
        &\leq 2|\dot{e}_h^n|_{1,h}^2 + 24C_7h^2 \left( \Vert \Delta v^n \Vert_{L^2}^2 + \Vert v^n \Vert_{L^2}^2 \right)
    \end{align*}
    to conclude that there exist constants $C_{11}, C_{12}, C_{13}>0$ such that
    \begin{align*}
        &(\overline{e}_h^n - \overline{e}_h^{n-1}, \tilde{e}_h^n)_2 + \frac{1}{2} \Vert \tilde{e}_h^n \Vert_2^2 - \frac{1}{2}\Vert \tilde{e}_h^{n-1} \Vert_2^2 + \frac{1}{2}  \Vert \tilde{e}_h^n - \tilde{e}_h^{n-1}\Vert_2^2 + \frac{\tau}{4} |\dot{e}_h^n|_{1,h}^2 \\
        &\leq C_{11}\tau h^2\left(\Vert \Delta v^n \Vert_{L^2}^2 + \Vert v^n \Vert_{L^2}^2 + \Vert \Delta \nabla v^n \Vert_{L^2}^2 + \Vert \nabla v^n \Vert_{L^2}^2\right) + C_{12} \tau \Vert \tilde{e}_h^n\Vert_2^2 \\
        & \quad+ \int_{\Lambda} (g(v^{n-1}) - g(\hat{v}_h^{n-1})) \tilde{e}_h^n  \Delta_n W \, dx + \int_{\Lambda} (g(\hat{v}_h^{n-1}) - g(u_h^{n-1}))\tilde{e}_h^n  \Delta_n W \, dx.
    \end{align*}
    %%%%%%%%%%%%%
    Replacing $n$ by $m$ and summing over $m=1,...,N$ yields
    \begin{align}\label{eq 051225_01}
            &\sum\limits_{m=1}^n(\overline{e}_h^m - \overline{e}_h^{m-1}, \tilde{e}_h^m)_2 + \frac{1}{2} \Vert \tilde{e}_h^n \Vert_2^2 + \frac{1}{2}  \sum\limits_{m=1}^n\Vert \tilde{e}_h^m - \tilde{e}_h^{m-1}\Vert_2^2 + \frac{\tau}{4} \sum\limits_{m=1}^n|\dot{e}_h^m|_{1,h}^2 \notag \\
            &\leq \frac{1}{2}\Vert \tilde{e}_h^0 \Vert_2^2 + C_{11}\tau h^2 \sum\limits_{m=1}^n\left(\Vert \Delta v^m \Vert_{L^2}^2 + \Vert v^m \Vert_{L^2}^2 + \Vert \Delta \nabla v^m \Vert_{L^2}^2 + \Vert \nabla v^m \Vert_{L^2}^2\right)  \notag \\
            & \quad + C_{12} \tau \sum\limits_{m=1}^n\Vert \tilde{e}_h^m\Vert_2^2 + \sum\limits_{m=1}^n\int_{\Lambda} (g(v^{m-1}) - g(\hat{v}_h^{m-1})) \tilde{e}_h^m  \Delta_m W \, dx \notag \\
            &\quad + \sum\limits_{m=1}^n\int_{\Lambda} (g(\hat{v}_h^{m-1}) - g(u_h^{m-1}))\tilde{e}_h^m  \Delta_m W \, dx.
    \end{align}  
    By using Lemma \ref{Lemma 260923_01} and \ref{Theorem 110124_01}, we get
    \begin{align*}
        |(\overline{e}_h^m - \overline{e}_h^{m-1}, \tilde{e}_h^m)_2| &\leq \frac{1}{2\tau} \Vert \overline{e}_h^m - \overline{e}_h^{m-1}\Vert_{L^2}^2 + \frac{\tau}{2} \Vert \tilde{e}_h^m\Vert_{L^2}^2 \\
        &\leq \frac{C_7}{2}\frac{h^2}{\tau} \Vert v^m - v^{m-1} \Vert_{H^2}^2 + \frac{\tau}{2} \Vert \tilde{e}_h^m\Vert_{L^2}^2 \\
        &\leq 6C_7 \frac{h^2}{\tau} \left(\Vert v^m - v^{m-1} \Vert_{L^2}^2 + \Vert \Delta v^m - \Delta v^{m-1} \Vert_{L^2}^2 \right) + \frac{\tau}{2} \Vert \tilde{e}_h^m\Vert_{L^2}^2, 
    \end{align*}
    and we may deduce from \eqref{eq 051225_01} that
    \begin{align}\label{eq 051225_02}
        \begin{aligned}
            &\left(\frac{1}{2} - (C_{12} + \frac{1}{2}) \tau\right) \Vert \tilde{e}_h^n \Vert_2^2 + \frac{1}{2}  \sum\limits_{m=1}^n\Vert \tilde{e}_h^m - \tilde{e}_h^{m-1}\Vert_2^2 + \frac{\tau}{4} \sum\limits_{m=1}^n|\dot{e}_h^m|_{1,h}^2 \\
            &\leq \frac{1}{2}\Vert \tilde{e}_h^0 \Vert_2^2 + C_{11}\tau h^2 \sum\limits_{m=1}^n\left(\Vert \Delta v^m \Vert_{L^2}^2 + \Vert v^m \Vert_{L^2}^2 + \Vert \Delta \nabla v^m \Vert_{L^2}^2 + \Vert \nabla v^m \Vert_{L^2}^2\right) \\
            & \quad+ C_{12} \tau \sum\limits_{m=1}^n\Vert \tilde{e}_h^{m-1}\Vert_2^2 + 6C_7 \frac{h^2}{\tau} \sum\limits_{m=1}^n\left(\Vert v^m - v^{m-1} \Vert_{L^2}^2 + \Vert \Delta v^m - \Delta v^{m-1} \Vert_{L^2}^2 \right)\\
            & \quad+ \sum\limits_{m=1}^n\int_{\Lambda} (g(v^{m-1}) - g(\hat{v}_h^{m-1})) \tilde{e}_h^m  \Delta_m W \, dx + \sum\limits_{m=1}^n\int_{\Lambda} (g(\hat{v}_h^{m-1}) - g(u_h^{m-1}))\tilde{e}_h^m  \Delta_m W \, dx.
        \end{aligned}
    \end{align}
    Taking the expectation in $V$ and applying Young's inequality yields
	\begin{align}\label{eq 011225_03}
        \begin{aligned}
            \erww{|V|} &\leq 2\tau \erww{\Vert g(v^{n-1}) - g(\hat{v}_h^{n-1})\Vert_2^2} + \frac{1}{8} \erww{\Vert \tilde{e}_h^n - \tilde{e}_h^{n-1}\Vert_2^2} \\
		&\quad + 2\tau \erww{\Vert g(\hat{v}_h^{n-1}) - g(u_h^{n-1})\Vert_2^2} + \frac{1}{8} \erww{\Vert \tilde{e}_h^n - \tilde{e}_h^{n-1}\Vert_2^2}.
        \end{aligned}
	\end{align}
    Hence, taking expectation in \eqref{eq 051225_02}, using Lemma \ref{Lemma 250923_01}, \ref{Lemma 250923_03} and Corollary \ref{Corollary 241125_01} and the assumptions on $g$ yields the existence of a constant $C>0$ such that
	\begin{align*}
		&\left(\frac{1}{2} - (C_{12} + \frac{1}{2}) \tau\right) \erww{\Vert \tilde{e}_h^n \Vert_2^2} + \frac{1}{4}  \sum\limits_{m=1}^n \erww{\Vert \tilde{e}_h^m - \tilde{e}_h^{m-1}\Vert_2^2} + \frac{\tau}{4} \sum\limits_{m=1}^n \erww{|\dot{e}_h^m|_{1,h}^2} \\
        &\leq \frac{1}{2} \erww{\Vert \tilde{e}_h^0 \Vert_2^2} + C_{11}\tau h^2 \sum\limits_{m=1}^n \erww{\Vert \Delta v^m \Vert_{L^2}^2 + \Vert v^m \Vert_{L^2}^2 + \Vert \Delta \nabla v^m \Vert_{L^2}^2 + \Vert \nabla v^m \Vert_{L^2}^2} \\
        &+ C_{12} \tau \sum\limits_{m=1}^n \erww{\Vert \tilde{e}_h^{m-1}\Vert_2^2} + 6C_7 \frac{h^2}{\tau} \sum\limits_{m=1}^n \erww{\Vert v^m - v^{m-1} \Vert_{L^2}^2 + \Vert \Delta v^m - \Delta v^{m-1} \Vert_{L^2}^2}\\
        &+ 2\tau \sum\limits_{m=1}^n \Vert \overline{e}_h^{m-1}\Vert_{L^2}^2 + 2\tau \sum\limits_{m=1}^n \Vert \tilde{e}_h^{m-1}\Vert_{L^2}^2 \\
        &\leq \frac{C}{2}\tau + C_{11}T(C_1 + C_2 + C_3 + 2C_7T)  h^2 + 6C_7(C_1 + C_3) \frac{h^2}{\tau}+ (2+C_{12}) \tau \sum\limits_{m=1}^n \erww{\Vert \tilde{e}_h^{m-1}\Vert_2^2}. 
	\end{align*}
    There exists $N_0 \geq \hat{N}_0$ such that $\left(1/2 - (C_{12} + 1/2) \tau\right) > 1/4$ for all $N \geq N_0$. Hence, for all $N \geq N_0$ and all $n \in \{1,...,N\}$ we have
    \begin{align*}
        &\erww{\Vert \tilde{e}_h^n \Vert_2^2} + \sum\limits_{m=1}^n \erww{\Vert \tilde{e}_h^m - \tilde{e}_h^{m-1}\Vert_2^2} + \tau \sum\limits_{m=1}^n \erww{|\dot{e}_h^m|_{1,h}^2} \\
        &\leq 2C\tau + 4C_{11}T(C_1 + C_2 + C_3 + 2C_7T)  h^2 + 12C_7(C_1 + C_3) \frac{h^2}{\tau} \\
        & \hspace{4cm} + 2(2+C_{12}) \tau \sum\limits_{m=1}^n \erww{\Vert \tilde{e}_h^{m-1}\Vert_2^2}.
    \end{align*}
    Now, applying Lemma \ref{Gronwall} finalizes the proof of this lemma.
\end{proof}

\begin{rem}
The constants $C_9$ and $C_{10}$ respectively depend on $C_7$ and this constant may depend on $\operatorname{reg}(\Tau_h)$. Assuming \eqref{mesh_regularity}, we may overcome this dependence.
\end{rem}

\section{Proof of Theorem \ref{main theorem}}
Let $t \in [0,T]$ and $N \geq N_0 \in \mathbb{N}$. Then there exists $n \in \{1,...,N-1\}$ such that $t \in [t_{n-1}, t_n)$ or $t \in [t_{N-1}, t_N]$. For $n \in \{1,...,N\}$ we have
\begin{align*}
&\erww{\Vert u(t) - u_{h,N}^r(t) \Vert_2^2} = \erww{\Vert u(t) - u_h^n \Vert_2^2}\\
&\leq \erww{\Vert u(t) - u(t_n) + u(t_n) - v^n + v^n - \hat{v}^n + \hat{v}^n - u_h^n \Vert_2^2} \\
&\leq 16 \left(\erww{\Vert u(t) - u(t_n) \Vert_2^2} + \erww{\Vert u(t_n) - v^n \Vert_2^2} + \erww{\Vert v^n - \hat{v}^n \Vert_2^2}  + \erww{\Vert \hat{v}^n - u_h^n \Vert_2^2} \right),
\end{align*}
where $v^n$ is the solution to \eqref{ES} with initial value $u_0$ and $\hat{v}^n$ is the centered projection of $v^n$. Now, Lemma \ref{Lemma 220923_01}, Lemma \ref{Lemma 220923_03}, Corollary \ref{Corollary 260923_01} and Lemma \ref{Lemma 270923_01} yield
\begin{align*}
&\erww{\Vert u(t) - u_{h,N}^r(t) \Vert_2^2} \\
&\leq 16(C_4 \tau + C_6 \tau + C_9 h^2 + C_{10} (\tau h^2 + \frac{h^2}{\tau})) \\
&\leq \Upsilon (\tau + h^2 + \frac{h^2}{\tau})
\end{align*}
for some constant $\Upsilon >0$ depending on the mesh regularity $\operatorname{reg}(\Tau_h)$ but not depending on $n, N$ and $h$ explicitly. If \eqref{mesh_regularity} is satisfied, $\Upsilon$ may depend on $\chi$ and the dependence of $\Upsilon$ on $\operatorname{reg}(\Tau_h)$ can be omitted. Therefore, taking the supremum over $t \in [0,T]$ on the left hand side and the right hand side of the inequality, we obtain
\begin{align*}
\sup_{t \in [0,T]} \mathbb{E} \Vert u(t) - u_{h,N}^r(t) \Vert_2^2 \leq \Upsilon (\tau + h^2 + \frac{h^2}{\tau}).
\end{align*}

\bibliographystyle{plain}
\bibliography{Rates_Sapountzoglou_Rajasekaran}

\end{document}